\documentclass[12pt]{amsart}

\usepackage{amsmath}
\usepackage{amsfonts}
\usepackage{amscd}
\usepackage{amssymb}
\usepackage{amsthm}
\usepackage{mathdots}
\usepackage{mathtools}

\usepackage{MnSymbol}
\usepackage[linktocpage=true]{hyperref}
\usepackage{mathrsfs}

\usepackage{tikz,tikz-cd, color}
\usepackage{adjustbox}

\usepackage{appendix}
\usetikzlibrary{matrix}
\usetikzlibrary{decorations.pathmorphing}

\tikzset{
  symbol/.style={
    draw=none,
    every to/.append style={
      edge node={node [sloped, allow upside down, auto=false]{$#1$}}}
  }
}

\usepackage{stmaryrd}
\usepackage{enumerate}
\usepackage{xcolor}

\usepackage[all]{xy}

\newtheorem{thm}{Theorem}[section]
\newtheorem{lem}[thm]{Lemma}
\newtheorem{prop}[thm]{Proposition}

\newtheorem{conj}[thm]{Conjecture}

\theoremstyle{definition}

\theoremstyle{remark}
\newtheorem{rem}[thm]{Remark}

\newcommand{\oo}{\mathcal{O}}
\newcommand{\bp}{\mathbb{P}}
\newcommand{\bc}{\mathbb{C}}
\newcommand{\Hom}{{\rm Hom}}
\newcommand{\lHom}{\mathcal{H}om}
\newcommand{\Ext}{{\rm Ext}}
\newcommand{\lExt}{\mathcal{E}xt}
\newcommand{\ext}{{\rm ext}}
\newcommand{\Quot}{{\rm Quot}}
\newcommand{\svk}{S_{V,k}}
\newcommand{\rk}{{\rm rk\,}}
\newcommand{\ch}{{\rm ch}}

\newcommand{\im}{{\rm im\,}}
\newcommand{\id}{{\rm id}}
\newcommand{\vir}{{\rm vir}}
\newcommand{\jbullet}{\mathbb{J}^\bullet}
\newcommand{\ozk}{{\mathcal{O}_{\mathscr{Z}}}}
\newcommand{\IT}{{\rm IT}}
\newcommand{\tr}{{\rm tr}}

\newcommand{\xrightarrowdbl}[2][]{%
  \xrightarrow[#1]{#2}\mathrel{\mkern-14mu}\rightarrow
}

\title{Rank-one sheaves and stable pairs on surfaces}
\author{Thomas Goller}
\address{Korea Institute for Advanced Study, Seoul}
\email{goller@kias.re.kr}

\author{Yinbang Lin}
\address{Yau Mathematical Sciences Center, Tsinghua University, Beijing}
\curraddr{School of Mathematical Sciences, Tongji University, Shanghai}
\email{yinbang.lin@icloud.com}

\date{}

\begin{document}

\maketitle

\begin{abstract}
    We study rank-one sheaves and stable pairs on a smooth projective complex surface.
    We obtain an embedding of the moduli space of limit stable pairs into a smooth space.
    The embedding induces a perfect obstruction theory \cite{LiTia98, BehFan97}, which, over a surface with irregularity $0$, agrees with the usual deformation-obstruction theory \cite{Lin18}.
    The perfect obstruction theory defines a virtual fundamental class on the moduli space.
    Using the embedding, we show that the virtual class equals the Euler class of a vector bundle on the smooth ambient space.
    As an application, we show that on $\mathbb{P}^2$, the expected count of the finite Quot scheme in \cite{BerGolJoh16} is its actual length.
    We also obtain a universality result for tautological integrals on the moduli space of stable pairs.
\end{abstract}

%\tableofcontents

\vskip40pt
\section{Introduction}

Grothendieck's Quot scheme, which parametrizes quotients of a fixed coherent sheaf, is a fundamental object in moduli theory.
Over projective curves, the study of Quot schemes has been very fruitful.
Using intersection theory on Quot schemes, Marian and Oprea provided a finite-dimensional proof of the Verlinde formula \cite{Mar07,MarOpr07a,MarOpr07b} and proved the strange duality conjecture over curves \cite{MarOpr07}.
One nice property of Quot schemes on curves is that when the fixed coherent sheaf is general, the Quot scheme is always of the expected dimension \cite{Gol18}.

Over a smooth projective surface, general Quot schemes are not as well-behaved, contrary to the nice example of Hilbert schemes of points.
Due to the possibility of torsion in the quotient, a Quot scheme on a surface will usually fail to be of the expected dimension.
Nevertheless, inspired by Marian and Oprea's study of strange duality over curves, Bertram, the first author, and Johnson \cite{BerGolJoh16} used Quot schemes of expected dimension 0 to provide evidence of Le Potier's strange duality conjecture for del Pezzo surfaces.
Also, in certain cases, the bad behavior of the Quot scheme can be offset by the existence of a virtual fundamental class \cite{Schultheis12, MarOprPan1701, OprPan19}.
Even when the Quot scheme is equipped with a virtual fundamental class, due to the complication of torsion, we do not know how to describe the class explicitly without using virtual localization (as in \cite{MarOprPan1701}).
However, Pandharipande and Thomas' study \cite{PanTho09, PanTho14} of curve counting invariants on Calabi-Yau 3-folds suggests that we can use \emph{stable pairs} to get around the issue of torsion.

Let $X$ be a smooth complex projective variety and $V$ be a fixed coherent sheaf over $X$.
A pair
\begin{equation*}
    (F,\alpha\colon V\to F),
\end{equation*}
consisting of a coherent sheaf $F$ and a morphism $\alpha$ from  $V$ to $F$, is \emph{(limit) stable} %Y: emphasize the term
if $F$ is pure and the dimension of the cokernel of $\alpha$ is strictly smaller than the dimension of $F$.\footnote{By the dimension of a coherent sheaf, we mean the dimension of its support.}
Compared to a quotient, we only require generic surjectivity for a stable pair, but we impose the purity condition on $F$.
The second author \cite{Lin18} showed that the moduli functor of equivalence classes of stable pairs is represented by a projective scheme, extending Koll\'ar's work \cite{Kol0805}.
\footnote{For small stability conditions on pairs, see \cite{Wan15}. Stable pairs from a trivial vector bundle were also studied in \cite{She11}.}

In this paper, we work in a setting similar to \cite{BerGolJoh16}, but we focus on stable pairs rather than the Quot scheme and consider more general surfaces. Let $X$ be a smooth projective surface over $\bc$ and $V$ be a locally free coherent sheaf on $X$. We let $S_{V,k}$ denote the moduli space of stable pairs of the form $(I_Z, \alpha \colon V \to I_Z)$, where $Z$ is a 0-dimensional subscheme of $X$ of length $k$ and $I_Z$ is its ideal sheaf. Our first result is an embedding of $\svk$ into a smooth ambient space. Namely, letting $\bp := \bp\Hom(V,\oo_X)$ denote the projective space of lines in $\Hom(V,\oo_X)$ and $X^{[k]}$ be the Hilbert scheme of $k$ points on $X$, there is a closed immersion
\[
    \iota = \iota_{\bp} \times \iota_{X^{[k]}} \colon S_{V,k} \hookrightarrow \bp \times X^{[k]},
\]
whose image is the zero locus of a section of the vector bundle
\[
    \pi_1^* \oo_{\bp}(1) \otimes \pi_2^* {V^*}^{[k]}
\]
(Proposition~\ref{zero-locus}).
Here, $\pi_1$ and $\pi_2$ are the projection maps on $\bp\times X^{[k]}$ and ${V^*}^{[k]}$ denotes the usual integral transform (for the definition, see \S\ref{hilb}). Even when $V$ is not locally free, we still have a similar statement (Remark \ref{zero-locus-nlf}).

Our main result is that, under certain conditions, the moduli space $S_{V,k}$ has a virtual fundamental class that can be explicitly described using the embedding.
\begin{thm}\label{vfc}
Let $X$ be a smooth projective surface and $V$ be a locally free sheaf over $X$.
Suppose the dual $V^*$ has no higher cohomology.
Then, 
\begin{enumerate}[(a)]
  \item the embedding $\iota$ induces a perfect obstruction theory, which,  under the condition (\ref{tr-def-surj}), agrees with the usual deformation-obstruction theory;
\item the perfect obstruction theory in turn induces a virtual fundamental class $[S_{V,k}]^{\vir}$, %Y: put in the notation for the virtual class
which agrees with the Euler class:
in $A_{\chi(V^*)-1-(\rk V-2)k}(\bp\times X^{[k]})$,
  \begin{equation*}
    \iota_* [S_{V,k}]^{\vir} = e(\pi_1^* \oo_{\bp}(1) \otimes \pi_2^* {V^*}^{[k]}).
  \end{equation*}
\end{enumerate}
\end{thm}

%\todo{T: I think it is not clear what expected dimension is meant. I would guess that the dimension coming from the deformation-obstruction theory is intended, but for that we would need condition (10).

%Y: You are right. I think we can remove this comment, and I have done so.}

For regular surfaces, $H^1(\oo_X)=0$, the condition (\ref{tr-def-surj}) is trivially satisfied.
By the usual deformation-obstruction theory, we mean (\ref{def-obs}), which is similar to that of the Quot scheme.

If we impose an additional condition on $V$, namely that the zero locus of every nonzero cosection $V \to \oo_X$ is 0-dimensional or empty,\footnote{Proposition \ref{P2-good-properties} provides evidence that this is a reasonable condition. This condition also appears in \cite{BerGolJoh16}.} then we get a close relationship between $S_{V,k}$ and the Quot scheme $\Quot(V,k)$ that parametrizes quotients $V \twoheadrightarrow F$ in which $F$ has rank 1, determinant $\oo_X$, and second Chern class $k$. In particular, in the case when $\Quot(V,k)$ is 0-dimensional, we have an isomorphism $S_{V,k} \cong \Quot(V,k)$ (Proposition \ref{svk-equals-quot}), and in this case the Euler class in the theorem computes the length of the Quot scheme
\[
    \# \Quot(V,k) = \int_{X^{[k]}} c_{2k} (V^{*{[k]}}) \tag{$\dagger$}
\]
(Proposition \ref{quot-count}).

We then apply these results to the enumerative problem studied in \cite{BerGolJoh16} for del Pezzo surfaces and pursued further by Johnson in \cite{Joh18}. In the setting when $H^1(\oo_X)=0$ and $\Quot(V,k)$ is expected to be finite, the duals $E$ of the kernels of the maps $V \twoheadrightarrow F$ in the Quot scheme satisfy the orthogonality condition $\chi(X,E \otimes F)=0$, so there is an induced determinant line bundle $\oo_{X^{[k]}}(\Theta_e)$ on $X^{[k]}$ that depends only on $e = \ch(E)$.
%\todo{Y: Does it matter whether it is Chow or cohomology?
%T: The determinant line bundle induced by $E$ depends on rank, determinant, and Euler characteristic of $E$ (HL, comment after Theorem 8.1.6).}
As shown in \cite{BerGolJoh16} for del Pezzo surfaces, one expects that if the Quot scheme is indeed finite and if the determinant line bundle is sufficiently positive, then $\# \Quot(V,k) = \chi(X^{[k]}, \oo_{X^{[k]}}(\Theta_e)) = h^0(X^{[k]},\oo_{X^{[k]}}(\Theta_e))$, but the method used in \cite{BerGolJoh16} to compute the length of the Quot scheme was not fully rigorous and was restricted to $k \le 7$. For surfaces $X$ with $\chi(\oo_X)=1$, Johnson checked that
\[
    \int_{X^{[k]}} c_{2k}({V^*}^{[k]}) = \chi(X^{[k]},\oo_{X^{[k]}}(\Theta_e)) \quad \mbox{for\ }k \le 11 \tag{$\ddagger$}
\]
%\todo{
% Y: We need to be careful about this conjectural equality. It is not true for all surfaces. See MOP's paper, third version, section 3.2.
% T: You are absolutely right. Drew's argument uses $\chi(\oo_X)=1$.
% Y: Did he check it for del pezzos or more surfaces?}
 by comparing universal factorizations of generating series obtained from these numbers, and conjectured that this agreement holds for all $k$ (Conjecture \ref{numbers-match}). Combining ($\dagger$) and ($\ddagger$) will give the desired equality, but we need to ensure that the conditions required for ($\dagger$) hold.

This raises a subtle question: in cases where the expected dimension is 0, is $\Quot(V,k)$ actually finite? This was proven in \cite{BerGolJoh16} for $\bp^2$ when $V$ is a general stable sheaf of rank $r \ge 2$ and the degree of $V$ is sufficiently negative (Theorem \ref{quot-finite}).
Other than the fact that $V$ will in general fail to be locally free when $r = 2$, such $V$ satisfies all of our assumptions, so we obtain
%the following improvement on Theorem A of \cite{BerGolJoh16} for $\bp^2$:

\begin{thm}\label{quot-equals-euler}Assume $1 \le k \le 11$, $r \ge 3$, and $d \gg 0$. Let $V$ be a general stable sheaf on $\bp^2$ of rank $r$, degree $-d$, and second Chern class chosen to ensure that the expected dimension of $\Quot(V,k)$ is 0. Then
\[
    \# \Quot(V,k) = h^0(X^{[k]},\oo_{X^{[k]}}(\Theta_e)).
\]
\end{thm}
This result improves on \cite[Theorem A]{BerGolJoh16} %Y: I put Theorem A in []
for $\bp^2$ by extending the range of $k$ from $k \le 7$ up to $k \le 11$ and making the count of the points of $\Quot(V,k)$ rigorous (see \S\ref{k-fold-points}).

A few comments about the conditions in the theorem are in order. Finiteness of the Quot scheme is also expected to be true for other del Pezzo surfaces, which would allow the theorem to be generalized beyond $\bp^2$. The condition $k \le 11$ appears because ($\ddagger$) has only been confirmed in that range; this condition can be dropped if Johnson's Conjecture~\ref{numbers-match} can be proved.

This result is interesting for Le Potier's strange duality conjecture because if the kernels produced by the Quot scheme could be shown to be semistable, then it would imply that the strange duality map is injective in these cases.

Finally, as another application of Theorem~\ref{vfc}, we study tautological integrals over $\svk$, obtaining a universality statement (Proposition~\ref{universality}).

During the preparation of our manuscript, we learned about the beautiful arXiv preprint \cite{OprPan19} by Oprea and Pandharipande.
Their paper contained a study similar to ours.
They studied dimension 0 and 1 quotients of a trivial vector bundle over a surface.
In particular, in the case of dimension 0 quotients, they obtained a closed formula for the generating series of the virtual Euler characteristics of the Quot schemes, via the virtual localization formula.

We organize the paper as follows.
In \S2, we review some basic facts about rank 1 sheaves, the Hilbert scheme of points and the moduli space of stable pairs.
In \S3, we show that the moduli space $\svk$ of stable pairs can be embedded into a smooth ambient space as the zero locus of a section.
In \S4, we show that the embedding induces a perfect obstruction theory and a virtual fundamental class.
Under some conditions on $V$, we compare the deformation-obstruction theory given by the embedding with the usual one, completing the proof of Theorem \ref{vfc}.
In \S5, we compare $S_{V,k}$ to the Quot scheme $\Quot(V,k)$ under an additional assumption on $V$. In particular, $S_{V,k}$ and $\Quot(V,k)$ are isomorphic when the Quot scheme is 0-dimensional. In \S6, we review Le Potier's strange duality conjecture and explain the relevance of Quot schemes for strange duality and how \cite{BerGolJoh16} used $k$-fold point formulas in an attempt to count the points of $\Quot(V,k)$. We then use Theorem \ref{vfc} to count the Quot scheme, which allows us to prove Theorem \ref{quot-equals-euler}. Finally, in \S7, we use Theorem~\ref{vfc} to express tautological integrals on $S_{V,k}$ as integrals on $X$, which is a point of independent interest.
\vskip10pt
{\it Acknowledgement.} The authors would like to thank Wanmin Liu and the Institute for Basic Science, Pohang, where the study was initiated. TG is grateful to the Korea Institute for Advanced Study for providing a good research environment, to Drew Johnson for useful discussions, and to Sreedhar Bhamidi and Bumsig Kim for their patient help. YL would like to thank the Korea Institute for Advanced Study. YL also would like to thank Cristina Manolache, Alina Marian, Dragos Oprea, and Baosen Wu for helpful discussions. YL is supported by Grants 2017M620726 and 2018T11083 from China Postdoctoral Science Foundation.

\vskip20pt
\section{Preliminaries}

Let $X$ be a smooth connected projective surface over $\mathbb{C}$.
Let $V$ be a fixed torsion free coherent sheaf on $X$.

\subsection{Sheaves of rank 1}

A rank 1 coherent sheaf on $X$ that is torsion free must be the tensor product of a line bundle and an ideal sheaf $I_Z$ of a 0-dimensional subscheme $Z$. When the rank 1 sheaf $F$ is not necessarily torsion free, an analysis of the possible torsion subsheaves $T$ of $F$ shows that
\begin{lem}[\cite{BerGolJoh16} Lemma 4.11]\label{rank1_classification}
Let $F$ be coherent sheaf on $X$ of rank 1 with trivial determinant and $c_2(F)=k \ge 0$. Then $F$ must be one of the following:
\begin{enumerate}
    \item $I_Z$, where $Z$ has length $k$;
    \item An extension $0 \to T \to F \to I_{Z} \to 0$, where $T$ has 0-dimensional support and $Z$ has length $>k$;
    \item An extension $0 \to T \to F \to \oo_X(-C) \otimes I_Z \to 0$, where $T$ has 1-dimensional support and $C > 0$ is an effective curve.
\end{enumerate}
\end{lem}

Since purity implies torsion freeness for sheaves of positive rank, we see that
\begin{lem}\label{pure-ideal-sheaf}
Let $F$ be a pure coherent sheaf on $X$ of rank 1 with trivial determinant and $c_2(F) = k \ge 0$.
Then $F \cong I_Z$ for $Z$ of length $k$.
\end{lem}

\subsection{Hilbert scheme of points}\label{hilb}
The Hilbert scheme of points on $X$, denoted $X^{[k]}$, parametrizes 0-dimensional subschemes of length $k$ on the surface $X$. It is an example of a Quot scheme. In contrast to general Quot schemes, the Hilbert scheme $X^{[k]}$ of points on a surface has particularly nice geometry.
In particular, $X^{[k]}$ is nonsingular of dimension $2k$ \cite{Fog68}.
The study of Hilbert schemes of points on surfaces is a vast field.
For an introduction, see the beautiful book \cite{Nak99}.

There is a universal closed subscheme $\mathscr{Z} \subset X^{[k]} \times X$ and a universal short exact sequence
\[
    0 \to I_{\mathscr{Z}} \to \oo_{X^{[k]} \times X} \to \oo_{\mathscr{Z}} \to 0.
\]
In particular, $\mathscr{Z}$ is flat and finite over $X^{[k]}$.
For a subscheme $Z\subset X$, the deformation space and the obstruction space are $\Hom(I_Z,\oo_Z)$ and $\Ext^1(I_Z,\oo_Z)$ respectively.
Here, $I_Z$ denotes the ideal sheaf and $\oo_Z$ denotes the structure sheaf.
Recall that the deformation space is naturally isomorphic to the Zariski tangent space.
There is a natural way to construct tautological sheaves on $X^{[k]}$ using $\oo_{\mathscr{Z}}$ as the kernel for an integral transform.
If $E$ is a locally free coherent sheaf on $X$, then we define
\[
    E^{[k]} := {\pi_1}_* (\pi_2^* E \otimes \oo_{\mathscr{Z}}),
\]
where $\pi_i$ are the projection maps on $X^{[k]} \times X$.
The fiber of $E^{[k]}$ at the point $Z \in X^{[k]}$ is $H^0(X,E \otimes \oo_Z)$, so the rank of the locally free sheaf $E^{[k]}$ is $k$ times the rank of $E$.

\subsection{Trace maps and calculations in the derived category}\label{trace-maps}

%{\ }\todo{T: I made some changes in this section because I wanted to introduce the more general $\lHom^{\bullet}(E^{\bullet},F^{\bullet})$ and explain the connection between the two trace maps using the hypercohomology spectral sequence (which we also need for Lemma \ref{diff-tr}). I also wrote another version of the proof of the lemma to match the later notion (e.g. using $\mathcal{A}^{\bullet}$). Let me know what you think.

%Y: I slightly modified the first few paragraphs. See the source file.}

We begin by mentioning a few useful facts about the derived category and then define the trace maps. For this part, see \cite{Art88, HuyLeh10}.

First, given a bounded complex of coherent sheaves $F^{\bullet}$ on a proper algebraic scheme $Y$, the hypercohomology groups $\mathbb{H}^i(F^{\bullet}) = R^i \Gamma(F^{\bullet})$ can be calculated by the spectral sequence
\[
    E_2^{i,j} = H^i(Y,\mathcal{H}^j(F^{\bullet})) \Rightarrow \mathbb{H}^{i+j}(F^{\bullet}),
\]
where $\mathcal{H}^j(F^{\bullet})$ denotes the $j$th cohomology sheaf of $F^{\bullet}$. A chain map between two such complexes induces a linear map between the hypercohomology groups.

Second, if $E^{\bullet}$ and $F^{\bullet}$ are finite complexes of coherent sheaves, consider the complex
\[
    \lHom^\bullet(E^\bullet, F^\bullet)
\]
defined by $\lHom^n(E^\bullet, F^\bullet)=\oplus_i \lHom(E^i,F^{i+n})$ and differential $d(\phi)=d_F\circ \phi - (-1)^{\deg \phi}\cdot \phi \circ d_E$. When $E^{\bullet}$ is a locally free resolution of a sheaf $E$ and $F^{\bullet}$ is an arbitrary resolution of a sheaf $F$, $\lHom^\bullet(E^\bullet, F^\bullet)$ is isomorphic to $R\lHom(E,F)$, and the spectral sequence mentioned above takes the following form: 
\[
    E_2^{i,j} = H^i(Y,\lExt^j(E,F)) \Rightarrow \Ext^{i+j}(E,F).
\]

Now, if $F$ is a non-trivial locally free coherent sheaf, there is a trace map $\tr_F\colon \lHom(F,F)\to \oo_Y$ and an inclusion $i_F\colon \oo_Y\to \lHom(F,F)$ sending $1$ to $\id_F$. 
Because $\tr_F\circ i_F=\rk F$, these maps induce a splitting of $\lHom(F,F)$.
By taking cohomologies of $\tr_F$, we obtain a trace map between cohomology groups $\tr_F\colon \Ext^i(F,F)\to H^i(\oo_Y)$
by abuse of notation.
Let $\Ext^i(F,F)_0$ denote its kernel, which is called the {\it traceless part}.
The splitting above also provides the splitting $\Ext^i(F,F)\cong \Ext^i(F,F)_0\oplus H^i(\oo_Y).$

For a general coherent sheaf $F$ with a finite locally free resolution $F^\bullet$, there are similar statements.
The complex $\lHom^\bullet(F^\bullet, F^\bullet)$, which represents $R\lHom(F,F)$, has a trace map
\begin{equation*}
    \tr_{F^\bullet}: \lHom^\bullet(F^\bullet, F^\bullet) \to \oo_Y
\end{equation*}
defined by setting $\tr_{F^\bullet}|_{\mathcal{E}nd(F^i)}=(-1)^i\tr_{F^i}$ and $\tr_{F^\bullet}|_{\lHom(F^i, F^j)}=0$ whenever $i\ne j$.
We obtain a trace map
\[
    \tr_F \colon \Ext^i(F,F) \to H^i(\oo_Y)
\]
as the induced map $\mathbb{H}^i(\tr_{F^{\bullet}})$. There is an inclusion of complexes
\begin{equation*}
    i_{F^\bullet}\colon \oo_Y \to \lHom^\bullet(F^\bullet, F^\bullet)
\end{equation*}
by sending $1$ to $\sum_i\id_{F^i}$.
Let $\rk F^\bullet=\sum_i (-1)^i\rk F^i$. 
Then, $\tr_{F^\bullet}\circ i_{F^\bullet}=\rk F^\bullet$.
When $F$ has positive rank, i.e. $\rk F^\bullet\not=0$, let $R\lHom(F,F)_0$ denote the kernel of $\tr_{F^\bullet}$. 
Then, combining $i_{F^{\bullet}}$ and the inclusion of this kernel gives a splitting
\begin{equation*}
     \oo_Y \oplus R\lHom(F,F)_0 \xrightarrow{\cong} \lHom^\bullet(F^\bullet, F^\bullet).
\end{equation*}
Applying $\mathbb{H}^i$ yields a splitting
\begin{equation*}
  H^i(\oo_Y) \oplus \Ext^i(F,F)_0 \xrightarrow{\cong} \Ext^i(F,F).
\end{equation*}
in which the traceless part $\Ext^i(F,F)_0$ is identified with $\mathbb{H}^i(R \lHom(F,F)_0)$.

A precise calculation of these splittings in the case of the universal ideal sheaf $I_{\mathscr{Z}}$ on $X^{[k]} \times X$ will be important for us.

\begin{lem}\label{end-splits}
$R\lHom(I_{\mathscr{Z}},I_{\mathscr{Z}}) \cong \oo \oplus \lHom(I_{\mathscr{Z}},\oo_{\mathscr{Z}})[-1]$.
\end{lem}
\begin{proof}
Because of the short exact sequence $0 \to I_Z \to \oo_X \to \oo_Z \to 0$, there is a distinguished triangle
\[
    R\lHom(I_{\mathscr{Z}},I_{\mathscr{Z}}) \to R\lHom(I_{\mathscr{Z}},\oo) \xrightarrow{\beta} R\lHom(I_{\mathscr{Z}},\oo_{\mathscr{Z}}) \to.
\]
We will compute the long exact sequence of cohomology.

Let $F^{\bullet} \to I_{\mathscr{Z}}$ be a two-step locally free resolution of $I_{\mathscr{Z}}$. Setting $F_i = (F^i)^*$, $\beta$ can be represented by the chain map
\[\xymatrix{
    \lHom^{\bullet}(F^{\bullet},\oo) \ar@{=}[r] & \{\, F_0 \ar[r]^{d} \ar[d] & F_{-1} \ar[d] \,\} \\
    \lHom^{\bullet}(F^{\bullet},\oo_{\mathscr{Z}}) \ar@{=}[r] & \{\, F_0 \otimes \oo_{\mathscr{Z}} \ar[r]^{\bar{d}} & F_{-1} \otimes \oo_{\mathscr{Z}} \,\}
}\]
The kernel of $d$ is $\lHom(I_{\mathscr{Z}},\oo) \cong \oo$ and the cokernel of $d$ is supported on $\mathscr{Z}$. Since the inclusion $\oo \hookrightarrow F_0$ vanishes on $\mathscr{Z}$, the induced map $\ker d \to \ker \bar{d}$ is 0. Since $\operatorname{coker} d$ is supported on $\mathscr{Z}$ and restriction is right-exact, the induced map $\operatorname{coker} d \to \operatorname{coker} \bar{d}$ is an isomorphism. Thus we see that \[
    \lHom(I_{\mathscr{Z}},I_{\mathscr{Z}}) \cong \lHom(I_{\mathscr{Z}},\oo) \cong \oo, \quad \lExt^1(I_{\mathscr{Z}},I_{\mathscr{Z}}) \cong \lHom(I_{\mathscr{Z}},\oo_{\mathscr{Z}}), \quad \lExt^2(I_{\mathscr{Z}},I_{\mathscr{Z}}) = 0.
\]

Now, representing $I_{\mathscr{Z}}$ by the complex
\[
    \mathcal{A}^{\bullet} = \{\, \oo \to \oo_{\mathscr{Z}} \,\}
\]
in degrees 0 and 1, there is a chain map
\[\xymatrix{           
    & \{ \, \ker d \ar@{^(->}[d] \ar[r]^-0 & \ker \bar{d} \, \} \ar@{^(->}[d] \\
    \lHom^{\bullet}(F^{\bullet},\mathcal{A}^{\bullet}) \ar@{=}[r] & \{ \, F_0 \ar[r] & (F_0 \otimes \oo_{\mathscr{Z}}) \oplus F_{-1} \ar[r] & F_{-1} \otimes \oo_{\mathscr{Z}} \, \}
}\]
that induces isomorphisms in cohomology. As $\lHom^{\bullet}(F^{\bullet},\mathcal{A}^{\bullet})$ represents $R\Hom(I_{\mathscr{Z}},I_{\mathscr{Z}})$ and $\ker d \cong \lHom(I_{\mathscr{Z}},\oo)$ and $\ker \bar{d} \cong \lHom(I_{\mathscr{Z}},\oo_{\mathscr{Z}})$ by construction, the proof is complete.
\end{proof}

%\todo{T: We only use this lemma in the appendix, so in principal we could move or remove it. But it fits well in the context of computing these splittings, so I like it here.

%Y: I like it here as well.
%}
Furthermore, we the following description of the traceless part.
\begin{lem}\label{diff-tr}
  The following composition is zero:
  \begin{equation*}
      \Hom(I_Z,\oo_Z) \to \Ext^1(I_Z,I_Z) \xrightarrowdbl{\tr} H^1(\oo_X).
  \end{equation*}
Moreover, $\Hom(I_Z,\oo_Z) \cong \Ext^1(I_Z,I_Z)_0$.
\end{lem}
It is clear by viewing them as tangent maps between moduli spaces. But we provide a direct proof.

\begin{proof} As the sheaves in Lemma \ref{end-splits} are flat over $X^{[k]}$, we get an analogous splitting
\[
    R\lHom(I_Z,I_Z) \cong \oo_X \oplus \lHom(I_Z,\oo_Z)[-1]
\]
for each $Z \in X^{[k]}$. Applying $\mathbb{H}^1$ to both sides yields a splitting
\[
    \Ext^1(I_Z,I_Z) \cong H^1(\oo_X) \oplus H^0(\lHom(I_Z,\oo_Z))
\]
identifying $\Hom(I_Z,\oo_Z) \cong H^0(\lHom(I_Z,\oo_Z))$ with the traceless part $\Ext^1(I_Z,I_Z)_0$.
\end{proof}

\subsection{Stable pairs}\label{intro-stable-pairs}
Fix an ample line bundle $\oo_X(1)$ on $X$.
A pair $(F,\alpha \colon V \to F)$ consisting of a coherent sheaf $F$ on $X$ and a morphism $V \to F$ is \emph{(limit) stable} if $F$ is pure and the dimension of the cokernel of $\alpha$ is strictly smaller than the dimension of $F$.
Two pairs $(F,\alpha \colon V \to F)$ and $(F',\alpha' \colon V \to F')$ are \emph{equivalent} if there is an isomorphism $\beta \colon F \to F'$ such that $\beta \circ \alpha = \alpha'$.
If we fix the rank, determinant, and second Chern class of $F$, then there is a projective fine moduli space $S$ parametrizing equivalent classes of stable pairs with these fixed invariants \cite[Theorem 1.1]{Lin18}.
In particular, there is a universal family on $S \times X$
\begin{equation*}
    \tilde{\alpha}\colon\pi_2^* V\to \mathcal{F}.
\end{equation*}

For a stable pair $(F,\alpha\colon V\to F)$, let $J^\bullet$ denote the complex $\{V\xrightarrow{\alpha} F\}$ where $V$ is in degree $0$ and $F$ is in degree $1$. There is a distinguished triangle
\begin{eqnarray}
\label{str-seq} J^\bullet \to V\xrightarrow{\alpha} F\to.
\end{eqnarray}
If we fix only the Hilbert polynomial of $F$ without fixing the determinant, the deformation space and the obstruction space are $\Hom(J^\bullet, F)$ and $\Ext^1(J^\bullet, F)$ respectively. But we are fixing the determinant, so we shall take the traceless part. It means the following. There are natural maps
\begin{equation*}
    \Ext^{i-1}(J^\bullet, F) \to \Ext^{i}(F,F) \xrightarrow{\rm tr} H^i(\oo_X).
\end{equation*}
We denote the kernel of the composition above by $\Ext^{i-1}(J^\bullet, F)_0$. Then, the deformation space and the obstruction space are
\begin{equation}\label{def-obs}
\Hom(J^\bullet, F)_0 \quad\mbox{and}\quad \Ext^1(J^\bullet, F)_0,
\end{equation}
respectively.

\begin{lem}The cohomology group $\Ext^{-1}(J^{\bullet}, F)$ is zero.
\end{lem}
\begin{proof}
Let $Q$ be the cokernel of $\alpha$.
Because $Q$ is a torsion sheaf and $F$ is torsion free, $\Hom(Q,F)=0$.
Hence, the morphism $\Hom(F,F)\to \Hom(V,F)$ is injective.
Therefore, $\Ext^{-1}(J^{\bullet}, F)=0$.
\end{proof}

Applying the functor $\Hom(-, F)$ to (\ref{str-seq}), we have the long exact sequence
\begin{equation}\label{def-str-seq}
\begin{array}{ccccccccc}
          0 & \to   & \Hom(F,F) & \xrightarrow{\circ \alpha} & \Hom(V,F)&\to & \Hom(J^{\bullet},F) \\
          & \to   & \Ext^1(F,F) &\to& \Ext^1(V, F)&\to &\Ext^1(J^{\bullet},F)\\
          & \to    & \Ext^2(F,F)&\to&\Ext^2(V, F)&\to &\Ext^2(J^{\bullet}, F) & \to & 0.
\end{array}
\end{equation}

In this paper, we focus on the case where $\rk F = 1$, $\det F \cong \oo_X$, and $c_2(F) = k \ge 0$.
Since $F$ is pure, Lemma \ref{pure-ideal-sheaf} implies that $F$ is an ideal sheaf $I_Z$ of a 0-dimensional subscheme $Z \subset X$ of length $k$.

\vskip20pt
\section{Embedding of $S_{V,k}$ into a smooth ambient space}

Let $X$ be a smooth connected complex projective surface and $V$ be a fixed torsion free coherent sheaf on $X$.
Let $S_{V,k}$ denote the moduli space of stable pairs $(I_Z,\alpha \colon V \to I_Z)$, where $Z$ is a 0-dimensional subscheme of $X$ of length $k$.

In this section, we show that $\svk$ can be embedded into a smooth ambient space.

\subsection{$S_{V,0}$ is a projective space}

When $k=0$, equivalence classes of stable pairs are evidently in bijection with points of the projective space
\[
    \bp := \bp \Hom(V,\oo_X)
\]
of lines in $\Hom(V,\oo_X)$, and this bijection is in fact an isomorphism of schemes. Another way of saying this is that one can view $\bp$ as the fine moduli space for families of nonzero maps $V \to \oo_X$.

\begin{prop}
$S_{V,0} \cong \bp$.
\end{prop}

\begin{proof}
There is a flat family $\pi_2^* V \to \pi_1^* \oo_{\bp}(1)$ on $\bp \times X$ parametrizing all maps $V \to \oo_X$. It is the composition
\[
    \pi_2^* V \to \Hom(V,\oo_X)^* \otimes \oo_{\bp \times X} \to \pi_1^* \oo_{\bp}(1)
\]
of the pull-backs of the natural map $V \to \Hom(V,\oo_X)^* \otimes \oo_X$ on $X$ and the universal map $\Hom(V,\oo_X)^* \otimes \oo_{\bp} \to \oo_{\bp}(1)$.
This family induces a morphism $\bp \to S_{V,0}$, which is clearly a bijection on closed points.
To show that this morphism is an isomorphism, it suffices to show that $S_{V,0}$ is smooth.

Taking $F = \oo_X$ in (\ref{def-str-seq}), the tangent space to $S_{V,0}$ at a closed point $(\oo_X, V \xrightarrow{\alpha} \oo_X)$ is the kernel of the connecting morphism $\delta$ in the exact sequence
\[
    0 \to \Hom(\oo_X,\oo_X) \xrightarrow{ \circ \alpha} \Hom(V,\oo_X) \to \Hom(J^{\bullet},\oo_X) \xrightarrow{\delta
    } \Ext^1(\oo_X,\oo_X).
\]
Here, we have identified $\Ext^1(\oo_X,\oo_X)$ with $H^1(X,\oo_X)$ via the trace map.
Since the kernel of $\delta$ is isomorphic to the cokernel of $\circ \alpha$, its dimension is independent of the choice of the closed point and equals $\dim S_{V,0} = \dim \bp$. Thus $S_{V,0}$ is smooth. (Note that this description of the tangent space is visibly the same as the tangent space of $\bp$.)
\end{proof}

\subsection{Description of $S_{V,k}$ as the zero locus of a section}
The above description of $S_{V,0}$ can be generalized to all $S_{V,k}$ as follows. We will realize the moduli space $S_{V,k}$ of stable pairs as the zero locus of a section of a coherent sheaf on a smooth scheme.

As above, let $\bp=\bp\Hom(V,\oo_X)$.
Over $\bp\times X$, we have the universal family
\begin{equation*}
\pi_2^*V \to \pi_1^*\oo_\bp(1),
\end{equation*}
and recall that there is a short exact sequence
 $ 0\to I_{\mathscr{Z}} \to \oo \to \oo_{\mathscr{Z}} \to 0$
on $X^{[k]} \times X$.
Over $\bp \times X^{[k]} \times X$, we have the following maps
\begin{equation}\label{univ-triple}
  \begin{tikzcd}
    &
    & V \arrow{d}{\phi}
    &
    &
    \\
    0 \arrow{r}
    & \oo_\bp(1)\otimes I_{\mathscr{Z}}\arrow{r}
    & \oo_\bp(1) \arrow{r}{\psi}
    & \oo_\bp(1)\otimes \oo_{\mathscr{Z}} \arrow{r}
    & 0
  \end{tikzcd}
\end{equation}
where the bottom row is exact. (We have omitted the notation for pull-backs.) Note that $\pi_1^*\oo_\bp(1)\otimes \pi_{23}^*\ozk$ is flat over $\bp\times X^{[k]}$.
Considering the cokernel of $\psi\circ \phi$, let $Z_0\subset \bp\times X^{[k]} $ be the closed subscheme, which is the flattening stratum with the maximum Hilbert polynomial $k$. Then, over $Z_0\times X$, the restriction of $\phi$ in (\ref{univ-triple}) factors through $\pi_1^*\oo_\bp(1)\otimes \pi^*_{23}I_{\mathscr{Z}}$.
This induces a family over $Z_0\times X$:
\begin{equation}\label{univ-fam}
  \tilde{\alpha}\colon \pi_2^* V\to
  (\pi_1^*\oo_\bp(1)\otimes \pi^*_{23}I_{\mathscr{Z}})|_{Z_0\times X}.
\end{equation}
(On the left, $\pi_2$ is the projection $Z_0 \times X \to X$, while on the right, $\pi_i$ denote the projection maps to the factors of $\bp \times X^{[k]} \times X$.) Moreover, assuming that $V$ is locally free, the morphism $\psi\circ \phi$
induces a section
\begin{equation*}
  \tilde{\sigma}\colon\oo_{\bp\times X^{[k]}\times X} \to \pi_3 ^*V^* \otimes \pi_1^*\oo_\bp(1)\otimes \pi_{23}^*\ozk.
\end{equation*}
Pushing forward this section along $\pi_{12}$,
we obtain
\begin{equation}\label{section}
    \sigma\colon \oo_{\bp\times X^{[k]}}\to \oo_{\bp}(1) \boxtimes {V^*}^{[k]}.
\end{equation}
By considering stalks, we know that the section $\sigma$ vanishes exactly along $Z_0$, so we denote $Z_0$ as $Z(\sigma)$.

\begin{rem}\label{zero-locus-nlf} In the case when $V$ is not locally free, we can still exhibit $Z_0$ as the zero locus of a section of a locally free sheaf. For this, we choose any quotient $\tilde{V} \twoheadrightarrow V$ with $\tilde{V}$ locally free. Then, replacing $V$ by $\tilde{V}$ and $\phi$ by the composition $\tilde{V} \twoheadrightarrow V \xrightarrow{\phi} \pi_1^* \oo_{\bp}(1)$ (which has the same image as $\phi$), the same construction will exhibit $Z_0$ as the zero locus of a section of $\oo_{\bp}(1) \boxtimes {\tilde{V}^{*[k]}}$, and the following proposition will still hold. But we will consider only the locally free case in later sections.
\end{rem}

\begin{prop}\label{zero-locus}
The scheme $Z(\sigma)$ is isomorphic to $\svk$ and (\ref{univ-fam}) is the universal family.
\end{prop}
\begin{proof}
We will prove that $Z(\sigma)$ equipped with the universal family $\tilde{\alpha}$ represents the moduli functor $\frak{S}_{V,k}$ of stable pairs.

Let $\alpha_T\colon\pi_2^* V\to \mathcal{F}_T$ be a flat family of stable pairs parametrized by $T$, with the fixed invariants. Then $\mathcal{F}_T$ is a family of ideal sheaves of $k$ points flat over $T$. It induces a morphism $\eta_{X^{[k]}}\colon T\to X^{[k]}$. The sheaf $\mathcal{F}_T$ differs from the pull-back $(\eta_{X^{[k]}}\times \id_X)^*(I_{\mathscr{Z}})$ by a line bundle pulled back from $T$. Letting $L_T$ denote this line bundle on $T$, we have
\begin{equation*}
    \pi_2^*V\to \mathcal{F}_T\hookrightarrow \pi^*_1L_T
\end{equation*}
where the composition is nontrivial on each fiber over $T$.
Here, $\pi_1$ and $\pi_2$ denote the projection maps on $T\times X$. We therefore obtain a morphism $\eta_\bp\colon T\to \bp$ such that $L_T$ is the pull-back of $\oo_\bp(1)$. We thus have a morphism
\begin{equation*}
      \eta=\eta_{\bp}\times \eta_{X^{[k]}}\colon T\to  \bp\times X^{[k]}.
\end{equation*}
The restriction of $\alpha_T$ to a closed point of $T$ is a map $\alpha \colon V\to I_Z$, and the image of this closed point under $\eta$ is $([V\xrightarrow{\alpha} I_Z\hookrightarrow \oo_X], Z)$.
Let $\tilde{\eta} = \eta \times \id_X$.
From the definition of $\eta$, it is clear that $\tilde{\eta}^*(\psi\circ \phi)=\tilde{\eta}^*\psi\circ \tilde{\eta}^*\phi=0$ over fibers.
Thus, $\eta$ factors through $Z_0=Z(\sigma)$ according to the universal property of flattening stratification.
Moreover, $\alpha_T$ is the pull-back of $\tilde{\alpha}$ via $\tilde{\eta}$.

So, we have defined a map
\begin{equation*}
      \tau_T\colon \frak{S}_{V,k}(T) \to {\rm Mor}(T,Z(\sigma)).
\end{equation*}
There is also a map
\begin{equation*}
      \upsilon_T\colon {\rm Mor}(T,Z(\sigma))\to \frak{S}_{V,k}(T),
\end{equation*}
which is given by pulling back $\tilde{\alpha}$. We have seen that $\upsilon_T\circ \tau_T$ equals identity.

To complete the proof, we show that $\tau_T \circ \upsilon_T$ is the identity. Let $\eta_1\in {\rm Mor}(T,Z(\sigma))$ and $\eta_2=\tau_T\circ\upsilon_T(\eta_1)$. Then, their images under $\upsilon_T$ are the same. Namely, letting $\tilde{\eta}_i = \eta_i \times \id_X$, there is an isomorphism
\begin{equation}\label{ggprime} \tilde{\eta}_1^*(
  \oo_\bp(1)\otimes
  I_{\mathscr{Z}}|_{Z(\sigma)\times X})\cong
  \tilde{\eta}_2^*(
  \oo_\bp(1)\otimes
  I_{\mathscr{Z}}|_{Z(\sigma)\times X})
\end{equation}
commuting $\upsilon_T(\eta_1)$ and $\upsilon_T(\eta_2)$.
(We are omitting some pull-backs again.)
According to the next lemma, $\tilde{\eta}_1^*\oo_\bp(1)\cong \tilde{\eta}_2^*\oo_\bp(1)$, as they are the double duals of the sheaves in (\ref{ggprime}).
So, $\eta_1$ and $\eta_2$ induce the same morphism to $\bp$. Combined with (\ref{ggprime}), this implies that they induce the same morphism to $X^{[k]}$ as well.
Thus, $\eta_1=\eta_2$.
\end{proof}

\begin{lem}
   We have an isomorphism $\tilde{\eta}_1^*\oo_\bp(1)\cong \tilde{\eta}_2^*\oo_\bp(1)$ extending (\ref{ggprime}).
\end{lem}
\begin{proof}
  Over $T\times X$, we have short exact sequences
\begin{equation*}
    0 \to \tilde{\eta}_i^*(
\oo_\bp(1)\otimes
I_{\mathscr{Z}}|_{Z(\sigma)\times X})  \to \tilde{\eta}_i^*
\oo_\bp(1) \to \tilde{\eta}_i^*\oo_{\mathscr{Z}} \to 0,
\end{equation*}
for $i=1,2$. Applying the functor $\lHom(-,\oo_{T\times X})$, we have the exact sequences
\[
    0\to \tilde{\eta}_i^*\oo_\bp(-1) \to
    (\tilde{\eta}_i^*(\oo_\bp(1)\otimes
I_{\mathscr{Z}}|_{Z(\sigma)\times X}))^* \to \mathcal{E}xt^1(\tilde{\eta}_i^*\oo_{\mathscr{Z}},\oo_{T\times X}) \to 0.
\]
Because of (\ref{ggprime}), it is enough to show $\mathcal{E}xt^1(\tilde{\eta}_i^*\oo_{\mathscr{Z}},\oo_{T\times X})=0$.
Notice that $\mathcal{E}xt^j(\tilde{\eta}_i^*\oo_{\mathscr{Z}},\oo_{T\times X})$ has relative dimension $0$ over $T$.
According to Grothendieck-Verdier duality,
\begin{equation*}
    R\pi_{1*}R\lHom(\tilde{\eta}_i^*\oo_{\mathscr{Z}},\oo_{T\times X}\otimes\omega_X[2])
    \cong
    R\lHom(R\pi_{1*}(\tilde{\eta}_i^*\oo_{\mathscr{Z}}),\oo_T).
\end{equation*}
Since $R\pi_{1*}(\tilde{\eta}_i^*\oo_{\mathscr{Z}})$ is a locally free sheaf, the right side is quasi-isomorphic to a sheaf over $T$. The cohomology sheaves of the left side are isomorphic to
$\pi_{1*}\mathcal{E}xt^j(\tilde{\eta}_i^*\oo_{\mathscr{Z}},\oo_{T\times X}\otimes \omega_X[2])$.
Thus, by comparing degrees, we see that $\mathcal{E}xt^1(\tilde{\eta}_i^*\oo_{\mathscr{Z}},\oo_{T\times X})=0$.
\end{proof}

We use
\begin{equation*}
    \iota \colon \svk \to \bp\Hom(V,\oo_X)\times X^{[k]}
\end{equation*}
to denote the map constructed in the proof above, which is an embedding and induces the isomorphism between $\svk$ and $Z(\sigma)$.
We denote the two induced maps to the factors by $\iota_\bp$ and $\iota_{X^{[k]}}$.

In \cite{KooTho14a,KooTho14b}, the authors use a similar embedding to study the moduli space of stable pairs and related deformation-obstruction theories, with applications towards curve counting invariants.
Their stable pairs are of the form $(F,s\colon \oo_X\to F)$ where $F$ has dimension 1.

\vskip20pt
\section{The perfect obstruction theory and the virtual fundamental class}\label{section_pot_vfc}

Let $X$ be a smooth connected projective surface over $\mathbb{C}$ and $V$ be a locally free coherent sheaf. Morever, assume  $V^*$ has no higher cohomology. Then the existence of a virtual fundamental class for $S_{V,k}$ is predicted by the vanishing of the higher obstruction space
$\Ext^{2}(J^{\bullet}, I_Z)$ for every ideal sheaf $I_Z$ of $k$ points, which follows from
\begin{equation}\label{ext2-v-i}
    \Ext^2(V,I_Z)\cong H^2(V^*\otimes I_Z)\cong H^2(V^*)=0
\end{equation}
and (\ref{def-str-seq}).

In this section, we prove Theorem~\ref{vfc}.
In \S~\ref{pot-vfc}, we show that the embedding $\iota$ induces a 2-term complex of vector bundles concentrated at degrees $-1$ and 0
\[\mathcal{E}^{\bullet}=\{\mathcal{E}^{-1}\to
\mathcal{E}^0\}\]
and a natural morphism to the truncated cotangent complex of $\svk$. This morphism is a perfect obstruction theory and in turn defines a virtual fundamental class. In \S~\ref{exp-description}, we show that the dual of $\mathcal{E}^{\bullet}$ can be described intrinsically as a complex $\mathcal{G}^{\bullet}$, which is essentially the traceless part of $R\lHom(\mathbb{J}^{\bullet},\mathcal{F})$.
Finally, in \S~\ref{obs-rhom}, we impose the condition that
\begin{equation}\label{tr-def-surj}
    \mbox{the composition }\Hom(J^\bullet, I_Z) \to \Ext^{1}(I_Z,I_Z) \xrightarrow{\rm tr} H^1(\oo_X)\mbox{ is surjective}
\end{equation}
for all points of $S_{V,k}$ and deduce that $\mathcal{G}^{\bullet}$ encodes the deformation and obstruction spaces (\ref{def-obs}) of stable pairs.

In this section, we often omit notation for pull-backs along projections.

\subsection{The perfect obstruction theory and the virtual fundamental class}
\label{pot-vfc}
The embedding $\iota$ yields an exact sequence
\begin{equation}\label{sec-seq}
    (\oo_\bp(1) \boxtimes V^{*[k]} )^*|_{\svk} \to  \Omega_{\bp\times X^{[k]}}|_{\svk}\to \Omega_{\svk}\to 0.
\end{equation}

We describe more explicitly (the dual of) the first morphism.
Let $W$ be the geometric vector bundle associated to the locally free sheaf $\oo_\bp(1) \boxtimes V^{*[k]}$ on $\bp \times X^{[k]}$.
With slight abuse of notation, we view the section $\sigma$ in (\ref{section}) as a geometric section of the vector bundle $W$
over $\bp\times X^{[k]}$.
More concretely, on closed points, $\sigma$ sends
\begin{equation*}
    ([s\colon V\to \oo_X],[q\colon \oo_X\to \oo_Z]) \mapsto  q\circ s,
\end{equation*}
where the fiber of $W$
at $([s],[q])$ is $H^0(X,V^* \otimes \oo_Z) \cong \Hom(V,\oo_Z)$. Here, we do have unambiguous choices of representatives $s$ and $q$, which are given by $\phi$ and $\psi$ in (\ref{univ-triple}).
Let $\pi\colon W\to \bp \times X^{[k]}$ be the projection and $T_\pi W$ the relative tangent bundle.
Then, we have the exact sequence of tangent bundles on $W$
\begin{equation*}
    0\to T_\pi W \to TW \xrightarrow{d\pi} \pi^*T(\bp \times X^{[k]}) \to 0.
\end{equation*}
Restricted to the zero section of $W$, which we denote as $0(\bp \times X^{[k]})$, the sequence splits, providing a surjection $$TW|_{0(\bp \times X^{[k]})}\to T_\pi W|_{0(\bp \times X^{[k]})}\cong W.$$
The section $\sigma$ induces $T(\bp\times X^{[k]})\xrightarrow{d\sigma} TW$.
Then, in terms of sheaves, the first morphism in (\ref{sec-seq}) is the dual of the composition of the restrictions
\begin{equation}\label{dsigma-comp}
T(\bp\times X^{[k]})|_{Z(\sigma)}\xrightarrow{d\sigma|_{Z(\sigma)}}
T W|_{Z(\sigma)} \to \oo_{\bp}(1) \boxtimes {V^*}^{[k]}|_{Z(\sigma)}.
\end{equation}

We next explain that the composition (\ref{dsigma-comp}) can be obtained from some natural exact sequences.
The geometric section $\sigma$ is defined by compositions with $s\colon V\to \oo_X$ and $q\colon \oo_X \to \oo_Z$ on two components.
(Recall that, for a point in $Z(\sigma)$, $\alpha$ factors through $I_Z$, so the composition is 0.)
This is also the effect on tangent vectors.
Since the composition over $\bp\times X^{[k]}$
\[
    \sigma \colon \oo \to H^0(V^*) \otimes \oo_{\bp}(1) \to \oo_{\bp}(1) \boxtimes {V^*}^{[k]}
\]
is 0 on $Z(\sigma) = S_{V,k}$, there is an induced map
\[
    (H^0(V^*) \otimes \oo_{\bp}(1))/\oo|_{Z(\sigma)} \to  \oo_{\bp}(1) \boxtimes {V^*}^{[k]} |_{Z(\sigma)}.
\]
Furthermore, we can view it as
\begin{equation}\label{first-tangent-map}
     T \bp|_{Z(\sigma)} \to \oo_{\bp}(1) \boxtimes {V^*}^{[k]}|_{Z(\sigma)}.
\end{equation}
There is also a natural map
\[
    \pi_{1*}\lHom(I_{\mathscr{Z}},\oo_{\mathscr{Z}})|_{Z(\sigma)} \to \pi_{1*}\lHom(V,\oo_{\mathscr{Z}}) \otimes \oo_{\bp}(1)|_{Z(\sigma)}
\]
induced by $\pi_2^* V \to \mathcal{F}$, which we can view as
\begin{equation}\label{second-tangent-map}
    TX^{[k]}|_{Z(\sigma)} \to \oo_{\bp}(1) \boxtimes {V^*}^{[k]}|_{Z(\sigma)}.
\end{equation}
Then, the composition (\ref{dsigma-comp}) is the combination of (\ref{first-tangent-map}) and (\ref{second-tangent-map}).
To make this fact more transparent, the appendix includes a description of the restrictions of (\ref{dsigma-comp}) to fibers, using deformations over the dual numbers.

Let $ L_{\svk}$ denote the cotangent complex of $\svk$ and $t_{\geq -1}L_{\svk}$ be the truncation. Let $I$ be the ideal sheaf of $\svk$ in $\bp\times X^{[k]}$. The description of $\svk$ as the zero locus $Z(\sigma)$ provides a natural morphism $\phi$ of complexes concentrated in degrees $-1$ and $0$,
\begin{equation}\label{obs-thy}
    \begin{CD}
     \mathcal{E}^{\bullet} @= \{ (\oo_\bp(1) \boxtimes V^{*[k]})^*|_{\svk} @>>> \Omega_{\bp\times X^{[k]}}|_{\svk}\}\\
    @. @VVV @|\\
    @. \{I/I^2 @>>> \Omega_{\bp\times X^{[k]}}|_{\svk}\}.
    \end{CD}
\end{equation}
Moreover, the map $\mathcal{H}^0(\phi)$ between cohomology sheaves is an isomorphism and $\mathcal{H}^{-1}(\phi)$ is surjective. On the other hand, in the derived category ${\rm D}(\svk)$, the lower complex is isomorphic to $t_{\geq -1}L_{\svk}$, according to \cite[Corollaire 3.1.3, P205]{Ill71}. Thus, we have obtained a perfect obstruction theory for $\svk$ \cite[Definition 4.4, Definition 5.1]{BehFan97}. It follows from \cite[\S6 ``The basic example"]{BehFan97} that the push-forward to $\bp\times X^{[k]}$ of the corresponding virtual fundamental class is
\begin{equation*}
    \iota_* [\svk]^\vir=
    e(\oo_\bp(1) \boxtimes V^{*[k]}) \quad \mbox{in } A_{\chi(V^*)-1-(\rk V-2)k}(\bp\times X^{[k]}).
\end{equation*}

\begin{rem}
Let us consider taking an elementary modification $W$ of $V$ by a pure dimension 1 sheaf.
Then, $W$ is also locally free.
If $W$ has no higher cohomology either,
with this description of the virtual fundamental class, a formula relating the virtual fundamental classes is immediate.
\end{rem}

\subsection{A description of $\mathcal{E}^{\bullet}$ in terms of the universal sheaf}
\label{exp-description}
By Lemma \ref{end-splits} and the fact that $I_{\mathscr{Z}}$ is flat over $X^{[k]}$, %flatness ensures finite resolution
the trace map splits $R\lHom(\mathcal{F},\mathcal{F})$ as 
$$\oo_{Z(\sigma) \times X} \oplus \lHom(I_{\mathscr{Z}},\oo_{\mathscr{Z}})|_{Z(\sigma) \times X}[-1].$$ The sheaf $\lHom(I_{\mathscr{Z}},\oo_{\mathscr{Z}})$ is supported on $\mathscr{Z}$ and its pushforward to $X^{[k]}$ is known to be the tangent sheaf $TX^{[k]}$. 
Thus
\begin{equation*}
    R\pi_{1*}R\lHom(\mathcal{F},\mathcal{F})\cong R\pi_{1*}\oo_{\svk\times X} \oplus\iota_{X^{[k]}}^* TX^{[k]}[-1].
\end{equation*}
Let $t_{\geq 1}R\pi_{1*}\oo_{\svk\times X}$ denote the truncation, which has cohomologies $H^1(\oo_X)\otimes \oo_{S_{V,k}}$ and $H^2(\oo_X)\otimes \oo_{S_{V,k}}$ at degrees 1 and 2 respectively. Then, we have a distinguished triangle over $\svk$:
\begin{equation*}
    \oo_{S_{V,k}} \oplus \iota_{X^{[k]}}^* TX^{[k]}[-1] \to R\pi_{1*}R\lHom(\mathcal{F},\mathcal{F}) \to t_{\geq 1}R\pi_{1*}\oo_{\svk\times X}\to
\end{equation*}
Let $\jbullet= \{\pi_2^*V\xrightarrow{\tilde{\alpha}} \mathcal{F}\}$
be the complex over $\svk\times X$ positioned at degrees 0 and 1.
Applying $R\pi_{1*}R\lHom(-,\mathcal{F})$ to the distinguished triangle
\[
    \jbullet \to \pi^*_2 V \xrightarrow{\tilde{\alpha}} \mathcal{F} \to,
\]
we get a distinguished triangle
\begin{equation}\label{rel-hom-triangle}
    R\pi_{1*}R\lHom(\mathcal{F},\mathcal{F}) \xrightarrow{\hat{\alpha}} R\pi_{1*}R\lHom(\pi^*_2 V,\mathcal{F}) \to R\pi_{1*}R\lHom(\jbullet,\mathcal{F})\to.
\end{equation}
We consider the composition $R\pi_{1*}R\lHom(\jbullet,\mathcal{F})[-1]\to R\pi_{1*}R\lHom(\mathcal{F},\mathcal{F})\to t_{\geq 1}R\pi_{1*}\oo_{\svk\times X}$ and complete it to get the following distinguished triangle
\begin{equation}\label{define-Gbullet}
    \mathcal{G}^{\bullet}[-1]\to R\pi_{1*}R\lHom(\jbullet,\mathcal{F})[-1] \to t_{\geq 1}R\pi_{1*}\oo_{\svk\times X} \to.
\end{equation}
Furthermore, (\ref{rel-hom-triangle}) induces a distinguished triangle
\begin{equation}\label{def-g}
    %\lExt^1_{\pi_1}(\mathcal{F},\mathcal{F})_0[-1]\oplus \oo
    \oo_{S_{V,k}} \oplus \iota_{X^{[k]}}^* TX^{[k]}[-1]
    \to R\pi_{1*}R\lHom(\pi^*_2 V,\mathcal{F}) \to \mathcal{G}^\bullet \to.
\end{equation}
We will show in \S \ref{obs-rhom} that $\mathcal{G}^\bullet$ encodes the deformation and obstruction spaces (\ref{def-obs}) of stable pairs %when we assume 
under the condition (\ref{tr-def-surj}).
%\todo{T: We should make it clear that $\mathcal{G}^{\bullet}$ is the natural way to attempt to define a perfect obstruction theory for $S_{V,k}$.

%Y: I avoided this intentionally. We are not actually using Li-Tian to define a virtual fundamental class. I don't like the phrase "attempt to define".

%T: Right now, there is a sentence near the beginning of \S 4.3 that is quite similar to this.

%Y: The sentence has been modified.}

%\todo{T: We should be careful about using the word "complex" since we have not yet chosen a representative of $\mathcal{G}^{\bullet}$.}

In the following, we use $\pi_1,\pi_2$ to denote the projection maps on $Z(\sigma) \times X$. Unless otherwise specified, the subscheme $\mathscr{Z}$ will denote the pull-back of the universal subscheme to $\bp \times X^{[k]} \times X$ followed by the restriction to $Z(\sigma) \times X$.

\begin{prop}\label{complexes-qis}
The complex $\mathcal{G}^\bullet$ on $Z(\sigma)$ is isomorphic to %can be represented by
\[
    \mathcal{E}^{\bullet \vee} = \{ (T\bp \oplus TX^{[k]})|_{Z(\sigma)} \to \oo_{\bp}(1) \boxtimes {V^*}^{[k]}|_{Z(\sigma)} \}
\]
obtained by combining (\ref{first-tangent-map}) and (\ref{second-tangent-map}).
\end{prop}

\begin{proof} Let $\mathcal{F}^{\bullet} \to \mathcal{F}$ be a two-step locally free resolution of the universal sheaf $\mathcal{F} \cong \oo_{\bp}(1) \otimes I_{\mathscr{Z}}$ on $Z(\sigma) \times X \subset \bp \times X^{[k]}\times X$. We may assume $\mathcal{F}^{\bullet}$ is chosen such that $\tilde{\alpha} \colon V \to \mathcal{F}$ lifts to a chain map $V \to \mathcal{F}^0$ (replace $\mathcal{F}^0$ by $\mathcal{F}^0 \oplus V$, use $\tilde{\alpha}$ to define the map $\mathcal{F}^0 \oplus V \to \mathcal{F}$, and take the new kernel).
We also note that $\mathcal{F}$ is quasi-isomorphic to the complex $$\mathcal{B}^\bullet=\{\oo_{\bp}(1)\to \oo_{\bp}(1) \otimes \oo_{\mathscr{Z}}\}$$
in degrees 0 and 1. Then, the morphism $R\lHom(\mathcal{F},\mathcal{F}) \to R\lHom(V,\mathcal{F})$ can be represented by the chain map
\[
 \lHom^{\bullet}(\mathcal{F}^{\bullet},\mathcal{B}^{\bullet}) \to \lHom^{\bullet}(V,\mathcal{B}^{\bullet}).
\]
Combining this with the proof of Lemma \ref{end-splits}, there is a commutative diagram of chain maps
\begin{equation*}
    \begin{tikzcd}
      \oo\oplus \lHom(I_\mathscr{Z},\oo_\mathscr{Z})[-1] \ar{d}{\rm qis} \arrow{rd}{\beta} 
      & \\
      \lHom^\bullet(\mathcal{F}^\bullet, \mathcal{B}^\bullet) \arrow{r}{}
      & \lHom^\bullet(V, \mathcal{B}^\bullet)
    \end{tikzcd}
\end{equation*}
on $Z(\sigma) \times X$. Note that the composition $\beta$ is induced by the maps $\oo \cong \lHom(I_{\mathscr{Z}},\oo) \to \lHom(V,\oo_{\bp}(1)) \cong V^* \otimes \oo_{\bp}(1)$ and $\lHom(I_{\mathscr{Z}}, \oo_{\mathscr{Z}}) \to \lHom(V,\oo_{\bp}(1)\otimes \oo_{\mathscr{Z}}) \cong V^* \otimes \oo_{\bp}(1) \otimes \oo_{\mathscr{Z}}$ induced by $\tilde{\alpha}$.
  
Taking the (derived) push-forward of $\beta$ to $Z(\sigma)$ and taking the modification by the trace map (which amounts to removing the higher pushforwards of $\oo$) now yields a chain map
\begin{equation*}
    \begin{tikzcd}
      \{ \oo \arrow{r}{0} \arrow[hook]{d}
      & TX^{[k]}|_{Z(\sigma)}\} \arrow{d}\\
      \{ H^0(V^*)\otimes \oo_{\bp}(1) \arrow{r}
      & \oo_{\bp}(1)\boxtimes V^{*[k]}\},
    \end{tikzcd}
\end{equation*}
in which the left vertical map is the universal line from $\bp$ and the right vertical map is (\ref{second-tangent-map}). We can represent $\mathcal{G}^{\bullet}$ by the cone of this chain map, which is quasi-isomorphic to $\{ (T\bp \oplus TX^{[k]})|_{Z(\sigma)} \to \oo_{\bp}(1) \boxtimes {V^*}^{[k]}|_{Z(\sigma)} \}$ coming from (\ref{first-tangent-map}) and (\ref{second-tangent-map}).
\end{proof}

\subsection{Agreement of two obstruction theories}\label{obs-rhom}

Recall that, given a stable pair $(I_Z,\alpha\colon V\to I_Z)$ in $\svk$, the deformation and obstruction spaces are $\Hom(J^\bullet,I_Z)_0$ and $\Ext^1(J^\bullet,I_Z)_0$ respectively.
The vanishing (\ref{ext2-v-i}) implies that the following maps are surjective:
\begin{equation*}
    \Ext^{1}(J^\bullet, I_Z) \xrightarrowdbl{} \Ext^{2}(I_Z,I_Z) \xrightarrowdbl{\rm tr} H^2(\oo_X).
\end{equation*}
We now assume the condition (\ref{tr-def-surj}) holds.
Then the expected dimension of $S_{V,k}$, using these deformation and obstruction spaces, is given by
\begin{eqnarray}
&&\hom(J^\bullet,I_Z)_0-\ext^1(J^\bullet,I_Z)_0\nonumber\\
&=&\hom(J^\bullet,I_Z)-h^1(\oo_X)-\ext^1(J^\bullet,I_Z)_0+h^2(\oo_X)\nonumber\\
&=&\chi(J^\bullet,I_Z)+\chi(\oo_X)-1\nonumber\\
&=&\chi(V^*)-1-(\rk V-2)k
\label{sp_expdim}
%&=&h^0(V^*)-1+2k-k\rk V.
\end{eqnarray}
This is the same as the expected dimension from the deformation-obstruction theory (\ref{obs-thy}) given by the embedding.
%the expected dimension of the zero locus $Z(\sigma)$ and the degree of the virtual fundamental class from \S \ref{pot-vfc}.
%\todo{Y: I think the second half is redundant.

%T: Yes. I wrote that because we haven't explicitly mentioned the "expected dimension of the zero locus". Not sure about how to best write this...

%Y: I think it's OK to say the expected dimension of the zero locus. Or we can say "the expected dimension from the deformation-obstruction theory (\ref{obs-thy}) given by the embedding".

%Y: Changed.}

To complete the proof of Theorem \ref{vfc}, we show, assuming (\ref{tr-def-surj}), that the deformation and obstruction spaces are encoded in the complex $\mathcal{G}^\bullet$.
As before, let $J^\bullet$ denote the complex $\{V\xrightarrow{\alpha} I_Z\}$. 
To see that the cohomology of the (derived) restriction of $\mathcal{G}^\bullet$ to fibers at a point $p = (I_Z,\alpha \colon V \to I_Z)$ of $S_{V,k}$ agrees with the deformation and obstruction spaces, we use a simple base change argument. 
We consider the base change diagram
\[\xymatrix{
     \{ p \} \times X \ar[d]_{\pi} \ar@{^(->}[r]^{\tilde{j}} & S_{V,k} \times X \ar[d]^{\pi_1} \\
    \{ p \} \ar@{^(->}[r]_j & \svk
}.\]
Since $\pi_1$ is flat, the natural transformation of derived functors
\[
    Lj^* R{\pi_1}_* \to R{\pi}_* L\tilde{j}^*
\]
is an isomorphism \cite[Proposition 3.9.5, Theorem 3.10.3]{Lip09}. Applying this to $R\lHom(\jbullet,\mathcal{F})$ and using the fact that $R\lHom$ commutes with pull-backs \cite{Huy06}, we get
\begin{equation}\label{rhoms}
    Lj^* R{\pi_1}_* R\lHom(\jbullet,\mathcal{F}) \cong R\Gamma R\lHom(J^{\bullet},I_Z) \cong R\Hom(J^{\bullet},I_Z).
\end{equation}
The second isomorphism comes from $R\Gamma R\lHom(E^{\bullet},-) \cong R\Hom(E^{\bullet},-)$ for any $E^{\bullet}$, which follows from \cite[Proposition 2.58]{Huy06} since if $I^{\bullet}$ is a complex of injective sheaves, then $\lHom(E^{\bullet},I^{\bullet})$ is $\Gamma$-acyclic \cite{God58}.
Applying $Lj^*$ to the distinguished triangle (\ref{define-Gbullet}) and taking cohomology, we get the long exact sequence
\begin{multline*}
    0\to H^0(Lj^*\mathcal{G}^\bullet) \to \Hom(J^\bullet, I_Z) \to H^1(\oo_X)\xrightarrow{=0} \\
    H^1(Lj^*\mathcal{G}^\bullet) \to \Ext^1(J^\bullet, I_Z) \to H^2(\oo_X)\to 0.
\end{multline*}
The fourth arrow equals zero because of the condition (\ref{tr-def-surj}).
Therefore, 
\begin{equation*}
    H^i(Lj^*\mathcal{G}^\bullet)\cong \Ext^i(J^\bullet, I_Z)_0, \quad\mbox{for }i=0,1.
\end{equation*}

\vskip20pt
\section{Comparing moduli of stable pairs and Quot schemes}

%\todo{T: For some reason, I had included the condition (\ref{tr-def-surj}) in some of the statements in this section. But we don't seem to need it, so I removed it from the statements. Does that sound right to you?

%Y: I think we are now OK.}

In this section, we explore the relationship between the moduli of stable pairs $S_{V,k}$ and the Quot scheme $\Quot(V,k)$ (defined below). The connection between these two spaces is especially close if we assume that the zero locus of every nonzero map $V \to \oo_X$ is 0-dimensional, which we state more simply as
\begin{equation}\label{no-neg-quot}
  \text{$V$ has no cosections that vanish on curves}.
\end{equation}
To justify making such an assumption, we explain how to construct sheaves $V$ satisfying (\ref{no-neg-quot}), and we also ensure that the conditions assumed in previous sections ($V$ is locally free, $V^*$ has vanishing higher cohomology), but which are not necessary in this section, can be guaranteed.

But first, we make an observation about the condition (\ref{no-neg-quot}). Using the fact that any torsion free rank 1 sheaf on $X$ is the tensor product of a line bundle and an ideal sheaf of points, together with the fact that ideal sheaves of points and negative line bundles have inclusions into $\oo_X$, we easily see that

\begin{lem}\label{ideal-sheaf} The following conditions on $V$ are equivalent:
\begin{enumerate}
    \item $V$ has no cosections that vanish on curves.
    \item The image of any nonzero map $V \to I_Z$ is an ideal sheaf of a 0-dimensional subscheme.
    \item There is no nonzero morphism $V \to \oo_X(-D)$ for any effective divisor $D > 0$.
\end{enumerate}
\end{lem}

\subsection{Stable pairs and quotients}
Grothendieck \cite{Gro60} showed that for fixed $V$, there is a projective fine moduli space of quotients of $V$ with fixed invariants, called the \emph{Quot scheme}.
Let $\Quot(V,k)$ denote the Quot scheme whose closed points correspond to surjective maps $V \twoheadrightarrow F$,
where $F$ is a coherent sheaf with $\rk(F)=1$, $\det(F) = \oo_X$, and $c_2(F) = k$.
Let $U_{V,k} \subset \Quot(V,k)$ denote the open subscheme of quotients $V \twoheadrightarrow F$ such that $F$ is pure.
Then $U_{V,k}$ is isomorphic to the open subscheme of $S_{V,k}$ consisting of stable pairs in which the map is surjective (we denote this subscheme by $U_{V,k}$ as well).
Thus $S_{V,k}$ and $\Quot(V,k)$ can be viewed as two (usually different) compactifications of $U_{V,k}$.
%If the higher obstructions for the Quot scheme vanish (for instance, if $\Ext^2(V,F)$ vanishes for all quotients $F$), then $\Quot(V,k)$ has the same expected dimension as $S_{V,k}$.
%, as computed in (\ref{sp_expdim}).
%\todo{T: I modified this comment about expected dimension.

%Y: Did you remove the part after \%? Then, we need the condition (\ref{tr-def-surj})?

%T: I removed the comment entirely. It feels too subtle and I think there is no need to discuss the deformation-obstruction theory of the Quot scheme.

%Y: OK.}

One advantage of working with the Quot scheme is that its recursive structure can be exploited to prove results like

\begin{lem}
Let $r = \rk V \ge 2$ and suppose every component of $\Quot(V,k)$ has dimension $\le d$. Then $\Quot(V,k+\ell)$ is empty for all $\ell > \tfrac{d}{r}$.
\end{lem}

\begin{proof}
Assume for contradiction that there is a quotient $\beta \colon V \twoheadrightarrow F$ with $\ch(F) = (1,0,-k-\ell)$ and $\ell > \tfrac{d}{r}$. Since $F$ has rank one, its fibers are generically one-dimensional. Let $Z$ be any collection of $\ell$ distinct points at which the fibers of $F$ are one-dimensional. For each $p \in Z$, let $K_p$ denote the kernel of $\beta(p) \colon V(p) \twoheadrightarrow F(p)$. Then quotients $K_p \twoheadrightarrow \bc$ are in bijection with surjective maps $V(p) \twoheadrightarrow F(p) \oplus \bc$ that extend $\beta(p)$, and these quotients are parametrized by $\bp(K_p^*) \cong \bp^{r-2}$. Thus, for each $Z$, there is an $\ell(r-2)$-dimensional family of quotients $V \twoheadrightarrow F \oplus \oo_Z$. But there is a $2\ell$-dimensional family of choices of $Z$, so altogether we get an $\ell r$-dimensional family of quotients (as a Grassmann bundle) that must be a subscheme of $\Quot(V,k)$. This is a contradiction.
\end{proof}

\begin{rem}\label{quot-empty}
Two immediate consequences of the lemma are (1) $\Quot(V,k)$ is empty for all $k \gg 0$ and (2) if $\Quot(V,k)$ is empty, then $\Quot(V,k+\ell)$ is empty for all $\ell \ge 0$.
\end{rem}

Now, assume that $V$ satisfies the condition (\ref{no-neg-quot}). This guarantees that for every closed point of $S_{V,k}$, the image of $V \to I_Z$ is an ideal sheaf of points. It also ensures that for every closed point $V \twoheadrightarrow F$ of $\Quot(V,k)$, $F$ is an extension of a zero-dimensional torsion sheaf by an ideal sheaf of points.
Thus every closed point in $S_{V,k}$ or $\Quot(V,k)$ induces a closed point of $U_{V,k+\ell}$ for some $\ell \ge 0$, as follows: for $S_{V,k}$, replace the target sheaf by the image; for $\Quot(V,k)$, replace the target sheaf by its torsion-free quotient. Combining this observation with the previous lemma, we conclude the following:

\begin{prop}\label{svk-equals-quot} Suppose $V$ has no cosections that vanish on curves.
\begin{enumerate}[(a)]
\item The moduli space $S_{V,k}$ is empty if and only if $\Quot(V,k)$ is empty.
\item If $\Quot(V,k+1)$ (equivalently, $S_{V,k+1}$) is empty, then $S_{V,k} \cong \Quot(V,k)$.
\item If $\Quot(V,k)$ is 0-dimensional, then $S_{V,k} \cong \Quot(V,k)$.
\end{enumerate}
\end{prop}

\begin{proof}

(a) As mentioned above, a closed point of $S_{V,k}$ or $\Quot(V,k)$ induces a point of $U_{V,k+\ell}$ for some $\ell \ge 0$. But such a point can be used to construct a point of both $S_{V,k}$ (by embedding one ideal sheaf in another) and $\Quot(V,k)$ (by adding torsion to the quotient).

(b) Every closed point in $S_{V,k}$ or $\Quot(V,k)$ that is in the complement of $U_{V,k}$ induces a point of $U_{k+\ell}$ for some $\ell > 0$, but the assumption, (a), and Remark \ref{quot-empty} ensure that $U_{k+\ell}$ is empty for all $\ell > 0$. Thus $S_{V,k} \cong U_{V,k} \cong \Quot(V,k)$.

(c) By the previous lemma, $\Quot(V,k+\ell)$ is empty for all $\ell > 0$. Then apply (b).
\end{proof}

\subsection{Constructing $V$ with good properties}

Since local freeness is one of the conditions we imposed on $V$ in \S\ref{section_pot_vfc}, we will state the results in this subsection in terms of the dual sheaf $V^*$. Note that $V$ has no cosections vanishing on curves if and only if $V^*$ has no sections vanishing on curves.

\begin{prop} Let $X$ be a surface such that $h^1(\oo_X)=h^2(\oo_X)=0$. Let $r \ge 2$, let $L$ be a globally-generated line bundle on $X$ with vanishing higher cohomology, and let $W$ be a general 0-dimensional subscheme of $X$ of length $w$. Define $V^*$ as a general extension
\[
    0 \to \oo_X^{r-1} \to V^* \to L \otimes I_W \to 0.
\]
Then:
\begin{enumerate}[(a)]
    \item if $w > \chi(L \otimes \omega_X)$, then there are non-split extensions and a general extension is locally free;
    \item if $w \le \chi(L)$, then $V^*$ has vanishing higher cohomology;
    \item if the linear system of curves in $\bp H^0(X,L)$ that contain $W$ has no additional base points, then $V^*$ is globally generated.
\end{enumerate}

\end{prop}

\begin{proof}
(a) Note that $\ext^1(L\otimes I_W,\oo_X) = h^1(X,L \otimes \omega_X \otimes I_W) = w-\chi(L \otimes \omega_X) > 0$ since $L \otimes \omega_X$ has vanishing higher cohomology by Kodaira vanishing and $W$ is general. Thus there are non-split extensions. Similarly, the inequality $w-1 \ge \chi(L \otimes \omega_X) \ge 0$ implies that $h^0(X,L \otimes \omega_X \otimes I_{W'}) = 0$ for any subscheme $W' \subset W$ of length $w-1$. Therefore, $(L \otimes \omega_X,W)$ satisfies the Cayley-Bacharach property. It follows that a general extension is locally free.

(b) Since $L$ has vanishing higher cohomology, $W$ is general, and $w \le \chi(L)$, $L \otimes I_W$ also has vanishing higher cohomology, and the same is true for $\oo_X$ by assumption. This implies that $V^*$ has vanishing higher cohomology.

(c) The assumptions ensure that $V^*$ is an extension of globally generated sheaves, hence is globally generated.
\end{proof}

This method of constructing $V^*$ is quite natural:
If $V^*$ is globally generated of rank $r$, then a general choice of $r-1$ sections of $V^*$ yields a short exact sequence $0 \to \oo_X^{r-1} \to V^* \to L \otimes I_W \to 0$ for some line bundle $L$ and some 0-dimensional subscheme $W \subset X$.

\begin{prop}
\label{P2-good-properties}
Assume $X=\bp^2$.
Let $r \ge 2$, $d \ge 1$, and
\[
    \binom{d+1}{2} \le w \le \binom{d+2}{2} - 3 + \epsilon, \qquad \epsilon = \begin{cases}
    1 & \text{if $d=1,2$;} \\
    0 & \text{otherwise}.
    \end{cases}
\]
Let $W \subset \bp^2$ be a general 0-dimensional subscheme of length $w$. Then a general extension
\[
    0 \to \oo_{\bp^2}^{r-1} \to V^* \to \oo_{\bp^2}(d) \otimes I_W \to 0
\]
is locally free and globally generated, has vanishing higher cohomology, and has no nonzero sections vanishing on curves.
\end{prop}

\begin{proof} Local freeness and vanishing higher cohomology follow immediately from the previous proposition. Nets of degree $d$ curves defined by imposing general base points have no further base points by Lemma \ref{nets-no-base-points} (this is true even for pencils when $d = 1,2$), so $V^*$ is globally generated. Tensoring the short exact sequence by $\oo_{\bp^2}(-1)$ and taking cohomology yields
\[
    H^0(X,V^*(-1)) \cong H^0(X,\oo_{\bp^2}(d-1) \otimes I_W),
\]
and the latter is 0 due to the condition $w \ge \binom{d+1}{2} = h^0(\bp^2,\oo_{\bp^2}(d-1))$. Thus $V^*$ has no sections that vanish on lines, hence it has no sections that vanish on curves.
\end{proof}

\begin{lem}[\cite{Cop95} 3.3]\label{nets-no-base-points} Let $\delta = \bp H^0(\bp^2,\oo_{\bp^2}(d))$ denote the complete linear series, which has dimension $s = \binom{d+2}{2}-1$. Let $p_1,\dots,p_{s-2}$ be general points in $\bp^2$. Then the linear series $\delta(p_1,\dots,p_{s-2})$ of curves of degree $d$ that contain all the $p_i$ has no other base points.
\end{lem}

\begin{proof}
Let $C$ be a general smooth curve of degree $d$ that contains all the $p_i$. Since $C$ is general and the $p_i$ are general in $\bp^2$, the $p_i$ are also general in $C$. The restriction $\delta_C$ of $\delta$ to $C$ is very ample on $C$ and has dimension $s-1$. If $\delta(p_1,\dots,p_{s-2})$ has an additional base point $q$, then $q \in C$, and for all divisors $D \in \delta_C$ containing all the $p_i$, $D$ also contains $q$. But since the $p_i$ are general in $C$, this implies that the scheme $V_{s-1}^{s-2}(\delta_C)$ of effective divisors in $C$ of degree $s-1$ that fail to impose independent conditions on $\delta_C$ has dimension $\ge s-2$. This contradicts a generalization of the trisecant lemma, which states that $\dim V_{e}^{e-1}(\delta_C) \le e-2$ for all $e < s$ (\cite{ArbCorGri85} L-1, p.152).
\end{proof}

When $X$ is a del Pezzo surface, the condition that $V^*$ should have no sections vanishing on curves can be guaranteed by choosing the lower bound on $w$ with respect to a basis of the effective cone of $X$. This is not difficult to do, but as it requires a case-by-case analysis, we have not provided precise bounds. A similar upper bound for $w$ could also be established for each del Pezzo surface.
%\todo{T: modify this comment since we are not focusing on del Pezzo surfaces?

%Y: I am OK with the current paragraph, since I am not sure how far we can generalize the construction.

%T: Okay.}

\vskip20pt
\section{Application to strange duality}

In this section, we review how counting the points of finite Quot schemes can be part of a strategy for proving Le Potier's strange duality conjecture. We then explain how $k$-fold point formulas were used in \cite{BerGolJoh16} to obtain ``expected'' counts of Quot schemes, and we discuss how the moduli spaces of stable pairs relate to the loci of $k$-fold points that appear in the construction in \cite{BerGolJoh16}. Finally, we use Theorem \ref{vfc} to count finite Quot schemes in this context, which, in combination with results in \cite{BerGolJoh16} and \cite{Joh18}, allows us to prove Theorem \ref{quot-equals-euler}. 

Unless otherwise mentioned, we impose the condition $H^1(\oo_X)=0$ on the smooth projective surface $X$ to simplify the exposition.

\subsection{Strange duality}

One source of motivation to count points of finite Quot schemes comes from Le Potier's strange duality conjecture.

%Let $X$ be a smooth projective surface over $\mathbb{C}$ such that $h^1(\oo_X)=0$.
%If $E$ is a coherent sheaf on $X$, we view its Chern character $e=\ch(E) \in H^{2*}(X,\mathbb{Q})$ as the triple of numbers given by its rank, degree, and second Chern class. 

Let $X$ be a smooth projective surface over $\mathbb{C}$. Let $e,f \in H^*(X,\mathbb{Q})$ be cohomology classes such that $\chi(e \cdot f) = 0$. For sheaves $E$ and $F$ with $\ch(E)=e$ and $\ch(F)=f$, this condition says that
\begin{equation}\label{orthogonal}
    \sum_{i=0}^2 (-1)^i h^i(E \otimes F)=0.
\end{equation}
Let $M(e)$ and $M(f)$ denote the moduli spaces of semistable coherent sheaves on $X$ with Chern characters $e$ and $f$, respectively. Suppose that these moduli spaces are nonempty and satisfy the following two conditions on pairs $(E,F) \in M(e) \times M(f)$:
\begin{enumerate}[(a)]
\item $h^2(E \otimes F) = 0$ and $\mathrm{Tor}^1(E,F) = \mathrm{Tor}^2(E,F)=0$ for all $(E,F)$ away from a subset of codimension $\ge 2$;
\item $h^0(E \otimes F)=0$ for some $(E,F)$.
\end{enumerate}
Then the ``jumping locus''
\[
    \Theta = \{([E],[F]) \in M(e) \times M(f) \mid h^0(X,E \otimes F) \ne 0 \}
\]
has the structure of a Cartier divisor. Then we can restrict $\Theta$ to general fibers $\pi_2^{-1}([F])$ and $\pi_1^{-1}([E])$ of the projection maps $\pi_i$ on $M(e) \times M(f)$ to define Cartier divisors $\Theta_F \subset M(e)$ and $\Theta_E \subset M(f)$. The corresponding line bundles, which are called \emph{determinant line bundles}, depend only on the Chern characters of $F$ and $E$, so they are independent of the choice of fiber and we denote them by $\oo_{M(e)}(\Theta_f)$ and $\oo_{M(f)}(\Theta_e)$.
%\footnote{In particular, if $h^1(\oo_X) = 0$, then the Picard group of $X$ is discrete, hence the Picard groups of the moduli spaces are expected to be discrete, which would ensure the independence of the choice of fiber.}
Then, by the seesaw lemma, there is an isomorphism
\[
    \oo_{M(e) \times M(f)}(\Theta) \cong \pi_1^* \oo_{M(e)}(\Theta_f) \otimes \pi_2^* \oo_{M(f)}(\Theta_e),
\]
so the K\"{u}nneth formula implies that
\[
    H^0(M(e) \times M(f),\oo_{M(e) \times M(f)}(\Theta)) \cong H^0(M(e),\oo_{M(e)}(\Theta_f)) \otimes H^0(M(f),\oo_{M(f)}(\Theta_e)).
\]
The divisor $\Theta$ determines a unique (up to scaling) section of this line bundle, which induces a linear map
\[
    \operatorname{SD}_{e,f} \colon H^0(M(f),\oo_{M(f)}(\Theta_e))^* \to H^0(M(e),\oo_{M(e)}(\Theta_f))
\]
called the \emph{strange duality map}. One can check that this map sends the evaluation map $\operatorname{ev}_{[F]}$ to a scalar multiple of the section $s_F$ corresponding to the divisor $\Theta_F$.

%Since the Picard group of $X$ is discrete, we expect the Picard groups of the moduli spaces to be discrete. In this case, 

\begin{conj}[Le Potier's strange duality, \cite{LP05}] If $\operatorname{SD}_{e,f}$ is nonzero, then it is an isomorphism.
\end{conj}

See \cite{LP05, Dan02, Sca07} for more details on strange duality and \cite{HuyLeh10} for a discussion of determinant line bundles.

Few cases of the strange duality conjecture have been proved, even on $\bp^2$. For a sample of results related to strange duality, consult
\cite{Dan01, Yua16, Yua17b, Yua18, GotYua19, MarOpr13, MarOpr14, BolMarOpr17}.

\subsection{Relevance of Quot schemes to strange duality}
The following discussion is based on \cite{BerGolJoh16}, but the idea to use finite Quot schemes to study strange duality (on curves) first appeared in \cite{MarOpr07}.

Assuming a general sheaf in $M(e)$ is locally free, let $V$ be a sheaf with Chern character $e^* + f$, where we assume $\chi(e \cdot f)=0$ as in the previous subsection. Consider the Quot scheme $\Quot(V,f)$ parametrizing quotients $V \twoheadrightarrow F$ with $\ch(F)=f$. A point of the Quot scheme yields a short exact sequence
\begin{equation}\label{quot-extension}
    0 \to \tilde{E} \to V \to F \to 0.
\end{equation}
The orthogonality condition $\chi(E \otimes F)=0$ implies that
\[
    \sum_{i=0}^2 (-1)^i \ext^i(\tilde{E},F)=0
\]
and (assuming $\Ext^2(\tilde{E},F)$ vanishes for all $\tilde{E}$ and $F$) that the expected dimension of $\Quot(V,f)$ is 0. When the kernels and quotients are semistable, vanishing of the alternating sum should in general imply vanishing of each of the terms (since $\Theta$ has codimension 1), so
in particular $\Hom(\tilde{E},F)$, the Zariski tangent space to the Quot scheme, should be 0. Thus we expect $\Quot(V,f)$ to be finite and reduced.\footnote{We can only hope for this expected behavior if $V$ is general in some sense. In particular, we usually need each extension (\ref{quot-extension}) to be nonsplit, which is only possible if $\ext^1(F,\tilde{E}) > 0$. Since we need $\ext^1(\tilde{E},F) = 0$, Serre duality suggests that this is likelier to work if $\omega_X$ is negative.}

Assume that $\Quot(V,f)$ is indeed finite and reduced. Assume that all of the kernels are locally free. Then, writing $\Quot(V,f) = \{ [0 \to E_i^* \to V \to F_i \to 0] \}$ and assuming that the $E_i$ and $F_i$ are all semistable, one can observe that $([E_i],[F_j])$ lies on $\Theta$ if and only if $i \ne j$, which implies that the $\operatorname{ev}_{[F_i]}$ are linearly independent and that the restricted map
\[
    \operatorname{SD}_{e,f} \colon \operatorname{span}\{ \operatorname{ev}_{[F_i]}\} \to \operatorname{span} \{ s_{F_i} \}
\]
is an isomorphism. Thus, in this setting, strange duality reduces to the numerical statement
\begin{conj}\label{sd-numerical} In the setting above, suppose $\Quot(V,f)$ is finite and reduced. Then
\[
    \#\Quot(V,f) = h^0(M(e),\oo_{M(e)}(\Theta_f)) = h^0(M(f),\oo_{M(f)}(\Theta_e)).
\]
\end{conj}

%As in \cite{BerGolJoh16},
We study the case when $f = \ch(I_Z)$, where $Z$ is a 0-dimensional subscheme of length $k$. In this case, $M(f) = X^{[k]}$, and we abbreviate the notation for the Quot scheme by writing $k$ instead of $f = \ch(I_Z)$.

\subsection{Quot schemes and $k$-fold points}\label{k-fold-points}

Motivated by strange duality, Bertram, the first author, and Johnson \cite{BerGolJoh16} attempted to count the points of the Quot schemes $\Quot(V,k)$ on del Pezzo surfaces in cases where these Quot schemes have expected dimension 0. If the Quot scheme is indeed 0-dimensional, then every quotient is an ideal sheaf, and one expects that every quotient is an ideal sheaf of a reduced scheme. Then each point $[V \twoheadrightarrow I_Z]$ corresponds to a section $\oo_X \to V^*$ that vanishes at the $k$ distinct points $Z$.

In order to count sections of $V^*$ that vanish at $k$ distinct points, the authors make use of $k$-fold point formulas. Assuming that $V^*$ is globally generated and has vanishing higher cohomology, there is a short exact sequence
\[
    0 \to G \to H^0(X,V^*) \otimes \oo_X \to V^* \to 0.
\]
The closed points of $\bp(G)$, the projective bundle of lines in the fibers of the locally free sheaf $G$, are of the form $([s],p)$, where $s$ is a section of $V^*$ and $p$ is a point at which $s$ vanishes. There is a map
\[
    g \colon \bp(G) \to \bp H^0(X,V^*), \quad ([s],p) \mapsto [s]
\]
obtained by forgetting the point of the base $X$, and a section $\oo_X \to V^*$ that vanishes at $k$ distinct points $Z$ corresponds exactly to a point of $\bp H^0(X,V^*)$ at which the fiber of $g$ consists of $k$ reduced points. A point of $\bp H^0(X,V^*)$ with this property is called a \emph{$k$-fold point} of $g$.

In general, if $g \colon M \to N$ denotes a holomorphic map of complex manifolds, then the \emph{set of $k$-fold points of $g$} is defined as
\[
	N_k:=\left\{q \in N \colon \parbox{20em}{\centering $q$ has exactly $k$ distinct preimages $p_1,\dots,p_k$ and $g$ is an immersion at each $p_i$}\right\} \subset N.
\]
There are universal formulas that compute the cohomology class $[\overline{N}_k]$ of the closure of $N_k$ (in the analytic topology) in terms of Chern classes of the map $g$, but the most general formulas require $g$ to be suitably generic in a real analytic sense, which tends to be impossible to check for a holomorphic or algebraic map (in the literature, it is common practice to assume genericity). These formulas can be explicitly computed for $k \le 7$ by a method of Marangell and Rimanyi \cite{MarRim10}, but there is a serious difficulty in moving past $k=7$ because the classification of the possible singularities of the map $g$ is no longer discrete.

Returning to the map $g \colon \bp(G) \to \bp H^0(X,V^*)$ defined above, the fact that $\Quot(V,k)$ has expected dimension 0 implies that the set of $k$-fold points of $g$ has expected dimension 0 (we use the fact that $V^*$ has vanishing higher cohomology to ensure that $\bp H^0(X,V^*)$ has the expected dimension). The authors compute $k$-fold point formulas for $k \le 7$ in this setting to obtain an \emph{expected} count of the sections of $V^*$ that vanish at $k$ points (the word ``expected'' relates to the fact that genericity of $g$ cannot be checked) and thus also an \emph{expected} count of the points of $\Quot(V,k)$.

\subsection{Stable pairs and $k$-fold points}

In this subsection, we explain a close relationship between the sets of $k$-fold points of the map $g \colon \bp(G) \to \bp H^0(X,V^*)$ from the previous subsection and the spaces of stable pairs $S_{V,k}$. In order to keep track of $k$ in the notation, we write $\iota_{\bp,k} \colon S_{V,k} \to \bp$ for the maps that were previously denoted $\iota_{\bp}$.

Since sections of $V^*$ that vanish at a point are dual to cosections of $V$ that vanish at a point, a simple first observation is that $S_{V,1} \cong \bp(G)$. Moreover, $g$ coincides with $\iota_{\bp,1} \colon S_{V,1} \to \bp$.

\begin{rem} More generally, suppose $V^*$ is $(k-1)$-very ample, namely that the map $H^0(X,V^*) \otimes \oo_X \to V^*$ induces a surjection $H^0(X,V^*) \otimes \oo_{X^{[k]}} \to {V^*}^{[k]}$ (for example, $0$-very ampleness is equivalent to global generation). Let $G_k$ denote the kernel of the latter map. Then one can show that $S_{V,k}$ is isomorphic to $\bp(G_k)$, the projective bundle of lines in the fibers of $G_k$.
\end{rem}

\begin{lem}
Over a closed point $[s \colon V \to \oo_X] \in \bp$ with $\im s = I_W$,
the fiber of $\iota_{\bp,k} \colon S_{V,k} \to \bp$  is isomorphic to $S_{I_W,k}$.
Furthermore, it is isomorphic to $W^{[k]}$, which is the Hilbert scheme of length $k$ closed subschemes of $W$.
\end{lem}

\begin{proof}
Set-theoretically, a point in the fiber of $\iota_{\bp,k}$ over $s$ is a pair $V \to I_Z$ such that the composition $V \to I_Z \hookrightarrow \oo_X$ is $s$. Since $I_W$ is the image of $s$, this induces a map $I_W \to I_Z$, which is exactly a point of
$S_{I_W,k}$.
Such a map of ideal sheaves must be an inclusion, so $Z \subset W$, and it is equivalent to a surjective map $\oo_W\twoheadrightarrow \oo_Z$.
Thus the fiber, $S_{I_W,k}$, and $W^{[k]}$ all coincide as sets. To obtain these identifications at the level of schemes, one can compare the fiber with $S_{I_W,k}$ using arguments similar to the proof of Proposition \ref{zero-locus}, and one can compare $S_{I_W,k}$ with $W^{[k]}$ by using the universal properties of stable pairs and of the Hilbert scheme.
\end{proof}

Applying the lemma to $\iota_{\bp,1}$, we see that a $k$-fold point of $\iota_{\bp,1}$ is exactly a cosection $s \colon V \to \oo_X$ whose image is an ideal sheaf of $k$ distinct points. Thus, the set of $k$-fold points of $\iota_{\bp,1}$ is the image $\iota_{\bp,k}(U_{V,k}')$, where $U_{V,k}'\subset S_{V,k}$ is the open subscheme of surjective maps $V \to I_Z$ where $Z$ is reduced. Moreover, we can conclude that:

\begin{prop} Assume $U_{V,k}'$ is dense in $S_{V,k}$. Then the set-theoretic image of $\iota_{\bp,k} \colon S_{V,k} \to \bp$ is equal to the closure of the set of $k$-fold points of $\iota_{\bp,1} \colon S_{V,1} \to \bp$.
\end{prop}

\begin{proof} We have already observed that the set of $k$-fold points of $\iota_{\bp,1}$ is equal to $\iota_{\bp,k}(U'_{V,k})$. Since $\iota_{\bp,k}$ is proper, it is closed as a map of the underlying topological spaces, hence $\overline{\iota_{\bp}(U_{V,k}')} = \iota_{\bp}(\overline{U_{V,k}'}) = \iota_{\bp}(S_{V,k})$ (for constructible sets, closure in the Zariski topology is the same as closure in the analytic topology).
\end{proof}

Since the image of $\iota_{\bp,k} \colon S_{V,k} \to \bp$ has a natural scheme structure, we can view the proposition as a way to equip the closure of the locus of $k$-fold points of $\iota_{\bp,1}$ with a scheme structure (assuming the denseness condition holds).

As we have noted, the cohomology class of the closure of the set of $k$-fold points is only known to be well-behaved if the map $\iota_{\bp,1}$ is suitably generic, which we are unable to check. Even if we know that the set of $k$-fold points is finite, we still cannot guarantee that the $k$-fold point formula will correctly count the points. On the other hand, the virtual class of $S_{V,k}$ is both computable and is guaranteed to agree with the fundamental class of $S_{V,k}$ when $S_{V,k}$ is smooth. In the next subsection, we use these virtual classes as a rigorous way to count the points of finite Quot schemes.

\subsection{Counting points of finite Quot schemes}

We now return to the case of a general surface $X$. Assuming that $V$ is locally free,
%and condition (\ref{tr-def-surj}) is satisfied (for example, if $h^1(\oo_X)=0$),
we have constructed a closed embedding $\iota = \iota_{\bp} \times \iota_{X^{[k]}} \colon S_{V,k} \to \bp \times X^{[k]}$ and realized the push-forward of $[S_{V,k}]^{\text{vir}}$ as
\[
    e(\oo_{\bp}(1){ \boxtimes V^*}^{[k]}) = \sum_{i=0}^{kr} \pi_1^*(h)^i \cdot \pi_2^*(c_{kr-i}({V^*}^{[k]})),
\]
where $r = \rk(V)$ and $h$ is the the hyperplane class on $\bp$. Motivated by the connection between $S_{V,k}$ and the locus of $k$-fold points of $\iota_{\bp,1}$ described above, we push forward this class to $\bp$ under the projection map $\pi_1$ and get
\[
    {\iota_{\bp}}_*[S_{V,k}]^{\text{vir}} = \left(\int_{X^{[k]}} c_{2k}({V^*}^{[k]}) \right) h^{k(r-2)}.
\]
In cases where $k(r-2) \le \dim \bp$, this number
\[
    N_{V,k} := \int_{X^{[k]}} c_{2k}({V^*}^{[k]})
\]
can be interpreted as the degree of the pushforward of the virtual class (which is a replacement for the degree of the $k$-fold point class considered in \cite{BerGolJoh16}).

In cases where $S_{V,k}$ is finite and reduced, we can use this formula to count its points. We thus obtain a rigorous statement of the connection between $N_{V,k}$ and counts of finite Quot schemes that was observed by Johnson in \cite{Joh18}.

\begin{prop}\label{quot-count} Suppose $V$ is locally free and has no cosections vanishing on curves
%that condition (\ref{tr-def-surj}) is satisfied
and that the expected dimension of $\Quot(V,k)$ is zero. If $\Quot(V,k)$ is indeed finite and reduced, then
\[
    \# \Quot(V,k) = \# S_{V,k} = \int_{X^{[k]}} c_{2k}({V^*}^{[k]}).
\]
\end{prop}

\begin{proof}
In this case $\Quot(V,k) \cong S_{V,k}$ (Proposition \ref{svk-equals-quot}), so $S_{V,k}$ is smooth of the expected dimension $0$. Thus the virtual class coincides with the fundamental class.
\end{proof}

\subsection{Proof of Theorem \ref{quot-equals-euler}}

In order to apply Proposition \ref{quot-count}, we need to know that the Quot scheme is finite and reduced in cases when the expected dimension is 0. The main result in this direction is on $\bp^2$, though it is expected to hold in greater generality.

\begin{thm}[\cite{BerGolJoh16} Theorem B]\label{quot-finite}
Let $k \ge 1$, $r \ge 2$, and $d \gg 0$. Let $V$ be a general stable vector bundle on $\bp^2$ of rank $r$, degree $-d$, and second Chern class chosen to ensure that the expected dimension of $\Quot(V,k)$ is 0. Then $\Quot(V,k)$ is finite and reduced and each quotient is the ideal sheaf of a reduced subscheme.
\end{thm}

A simple computation shows that in this case,
\[
    c_2(V^*) = \binom{d+2}{2} - (k-1)(r-2),
\]
which is greater than $\binom{d+1}{2}$ since $d \gg 0$. Thus, we can apply Proposition \ref{P2-good-properties} and (\cite{BerGolJoh16}, Corollary 5.4) to conclude that $V$ has no cosections vanishing on curves, hence the number of points of $\Quot(V,k)$ is equal to $\int_{X^{[k]}} c_{2k}({V^*}^{[k]})$ (Proposition \ref{quot-count}).

The next step is to relate these integrals $\int_{X^{[k]}} c_{2k}({V^*}^{[k]})$ to the Euler characteristics of the relevant determinant line bundles. For this, we need $X$ to satisfy the condition $\chi(\oo_X)=1$ (but not necessarily $H^1(\oo_X)=0$). As before, let $f$ denote the invariants of an ideal sheaf of $k$ points (rank 1, determinant $\oo_X$, second Chern class $k$),
let $V$ be any vector bundle such that $\Quot(V,k)$ has expected dimension 0, and let $e^*$ denote the invariants of the kernels of the surjective maps that are parametrized by $\Quot(V,k)$. The expected dimension of the Quot scheme is 0 exactly when $\chi(e \cdot f) = 0$, and in this setting, Johnson conjectured that
\begin{conj}[\cite{Joh18} Conjecture 3.1]\label{numbers-match} Let $X$ be a smooth projective surface with $\chi(\oo_X)=1$. Then
\[
    \int_{X^{[k]}} c_{2k}({V^*}^{[k]}) = \chi(X^{[k]},\oo_{X^{[k]}}(\Theta_e)).
\]
\end{conj}
%\todo{
% Y: We need to be careful about this conjectural equality. It is not true for all surfaces. See MOP's paper, third version, section 3.2.
 
% T: You are right. I've added the crucial condition $\chi(\oo_X)=1$.
% }
\noindent Johnson verified the conjecture computationally for $k \le 11$ (\cite{Joh18} mentions only $k \le 6$, in the context of del Pezzo surfaces) and it is expected to hold for all $k$.

Thus, combining Theorem \ref{quot-finite} for $\bp^2$, the known cases of Conjecture \ref{numbers-match} ($k \le 11$), and the fact that the higher cohomology of the determinant line bundles vanishes for $d \gg 0$ (\cite{BerGolJoh16} Corollary 2.3), we complete the proof of Theorem \ref{quot-equals-euler}.

Note that Theorem \ref{quot-equals-euler} gives the equality of the first and third quantities in Conjecture \ref{sd-numerical}. However, we do not know how to relate these quantities to the number of sections of the determinant line bundle on $M(e)$.

\begin{rem}
Conjecture~\ref{numbers-match} has been proved for Enriques surfaces, see \cite[Proposition 4]{MarOprPan1712}.
\end{rem}

\vskip20pt
\section{Application to tautological integrals}
In this section, as a direct application of Theorem \ref{vfc}, we study the virtual integrals of polynomials in tautological classes over the moduli space $\svk$. As before, let $\jbullet$ denote the complex $\{\pi_2^* V\xrightarrow{\alpha} \mathcal{F}\}$ on $S_{V,k} \times X$, positioned at degrees 0 and 1. As usual we use $\pi_1,\pi_2$ to denote the projection maps onto the two factors. For a vector bundle $\Lambda$ over $X$, we consider the integral transform
\begin{equation}\label{int-tr}
    \IT(\Lambda)=R\pi_{1*}(\pi_2^*\Lambda\otimes \jbullet).
\end{equation}
In the Grothendieck $K$-group $K(\svk)$, it equals
\begin{eqnarray*}
&& R\pi_{1*}\pi_2^*(\Lambda\otimes V)-R\pi_{1*}(\pi_2^*\Lambda\otimes \mathcal{F})\\
&=& \oo^{\oplus\chi(\Lambda\otimes V)}
-R\pi_{1*}(\pi_2^*\Lambda\otimes \pi_1^*\iota_\bp^*\oo(1)\otimes (\iota_{X^{[k]}}\times \id_X)^*I_{\mathscr{Z}})
\quad \mbox{(Proposition~\ref{zero-locus})}\\
&=& \oo^{\oplus\chi(\Lambda\otimes V)}
-\iota_\bp^*\oo(1)\otimes
R\pi_{1*}(\iota_{X^{[k]}}\times \id_X)^*(\pi_2^*\Lambda
\otimes I_{\mathscr{Z}}) \\
&=& \oo^{\oplus\chi(\Lambda\otimes V)}
-\iota_\bp^*\oo(1)\otimes \iota_{X^{[k]}}^*R\pi_{1*}(\pi_2^*\Lambda
\otimes I_{\mathscr{Z}})\\
&=& \oo^{\oplus\chi(\Lambda\otimes V)}
-\iota_\bp^*\oo(1)^{\oplus \chi(\Lambda)}
+\iota_\bp^*\oo(1)\otimes \iota_{X^{[k]}}^*\Lambda^{[k]}.
\end{eqnarray*}
So, $\IT(\Lambda)$ is a pull-back from $\bp\times X^{[k]}$.
If $P$ is a polynomial in the Chern classes of $\IT(\Lambda)$, then by Theorem \ref{vfc},
\begin{equation}\label{extend-integral}
  \int_{[\svk]^\vir}P =\int_{\bp\times X^{[k]}}P\cdot e(V^{*[k]}\boxtimes \oo_\bp(1)).
\end{equation}

The tautological sheaves $V^{*[k]}$ on $X^{[k]}$ and $V^{*[k+1]}$ on $X^{[k+1]}$ can be related via the incidence correspondence $X^{[k,k+1]}\subset X^{[k]}\times X^{[k+1]}$.
Set-theoretically, the correspondence is
$X^{[k,k+1]}=\{(Z,Z^\prime)\mid Z\subset Z^\prime\}$.
This allows one to apply induction on the number of points.
Applying an argument similar to the proof of \cite[Proposition 0.5]{EllGotLeh01} yields the following universality statement.
\begin{prop}\label{universality}
  We assume the same conditions on $X$ and $V$ as in Theorem~\ref{vfc}.
  There is a universal polynomial $\tilde{P}$, depending only on $P$, in variables $\rk V$, $\rk \Lambda$, and intersection numbers among  $c_i(X)$, $c_i(V)$, and $c_i(\Lambda)$, such that
  \begin{equation*}
    \int_{[\svk]^\vir}P=\tilde{P}.
  \end{equation*}
\end{prop}

\vskip20pt
\section*{Appendix}

We fix a point $(I_Z,\alpha\colon V\to I_Z)$ in $Z(\sigma)\cong \svk$. Let $s$ be the composition $V\xrightarrow{\alpha}I_Z \hookrightarrow \oo_X$. A base change argument using \cite[Theorem 1.4, Remark 1.5]{Lan83} or \cite{BPS80} shows that the map (\ref{dsigma-comp}) on fibers over $(I_Z,\alpha)$ is of the form
\begin{equation}\label{fibers-pot}
\Hom(V,\oo_X)/\mathbb{C} s \oplus \Hom(I_Z,\oo_Z)
\to \Hom(V,\oo_Z).
\end{equation}
We will show that (\ref{fibers-pot}) is obtained from the map $\Hom(V,\oo_X) \to \Hom(V,\oo_Z)$ induced by composition with $\oo_X \twoheadrightarrow \oo_Z$ as well as the map $\Hom(I_Z,\oo_Z) \to \Hom(V,\oo_Z)$ induced by composition with $\alpha$. This proves that the description of the global map (\ref{dsigma-comp}) given in \S\ref{pot-vfc} is correct since that description agrees with the fiberwise calculation given here on all fibers.

We will also show, given Condition (\ref{tr-def-surj}), that the kernel and the cokernel of (\ref{fibers-pot}) are $\Hom(J^\bullet, I_Z)_0$ and $\Ext^1(J^\bullet,I_Z)_0$, respectively.

Let $T={\rm Spec\,}\bc[\epsilon]/\epsilon^2$.
To describe (\ref{fibers-pot}), we will view spaces in (\ref{fibers-pot}) as Zariski tangent spaces and
describe elements in them as morphisms from $T$.
Furthermore, since $\bp$ and $X^{[k]}$ are fine moduli spaces, tangent vectors over $\bp\times X^{[k]}$ are equivalent to families over $T$.

An element $\lambda$ in $\Hom(I_Z, \oo_Z)$ corresponds to a (flat) deformation $q_T\colon \oo_{T\times X}\to Q$ of the quotient $\oo_X\to \oo_Z$ over $T$.
The quotient $q_T$ fits into the following commutative diagram of exact sequences over $T\times X$.
\begin{equation}\label{def-f}
  \begin{tikzcd}
    & 0 \arrow{d} & 0 \arrow{d} & 0 \arrow{d}&\\
    0 \arrow{r} & I_Z \arrow{r} \arrow{d} & \oo_X \arrow{r} \arrow{d}{\cdot \epsilon} & \oo_Z \arrow{r} \arrow{d} & 0\\
    0 \arrow{r} & E \arrow{r} \arrow{d} & \oo_{T\times X} \arrow{r}{q_T} \arrow{d}{/\epsilon} & Q \arrow{r} \arrow{d} & 0\\
    0 \arrow{r} & I_Z \arrow{r} \arrow{d} & \oo_X \arrow{r} \arrow{d} & \oo_Z \arrow{r} \arrow{d} & 0\\
    &0 &0 &0&
  \end{tikzcd}
\end{equation}
Pushing forward along the projection $T\times X\to X$, we view the diagram above as one over $X$. Then, $\oo_{T\times X}$ splits as $\oo_X\oplus \epsilon\oo_X$ on $X$. We consider the composition
$$I_Z\hookrightarrow \oo_X\hookrightarrow \oo_X\oplus \epsilon\oo_X \cong \oo_{T\times X} \xrightarrow{q_T} Q.$$
Here, the second inclusion is into the first summand. Viewing these maps as a path through the diagram from the $I_Z$ in the bottom-left corner to $Q$, we see that the composition is 0 modulo $\epsilon$. Thus, it factors through the $\oo_Z$ in the upper-right corner, inducing a map $I_Z\to \oo_Z$, which is $\lambda$.
(See, for example, \cite[Theorem 6.4.5]{FanGotIll05}.)
On the other hand, an element in $\Hom(V,\oo_X)/\bc s$ corresponds to a deformation
$$\bar{s}\colon \pi^*_2V \cong V\oplus \epsilon V\to \oo_{T\times X}\cong \oo_X\oplus \epsilon\oo_X$$ of $s$.
We can represent $\bar{s}$ by a matrix
$\bigl( \begin{smallmatrix}s & 0\\ a & s\end{smallmatrix}\bigr)$ by choosing a representative $a\in \Hom(V,\oo_X)$.

A tangent vector on $\bp\times X^{[k]} $ corresponds to the diagram (\ref{def-f}) and $\bar{s}$.
It also corresponds to a morphism $T\to \bp \times X^{[k]}$.
Then, the image of the tangent vector under $d\sigma$ is the composition
\begin{equation*}
  T\to \bp \times X^{[k]} \xrightarrow{\sigma} W,
\end{equation*}
which corresponds to the family
  $V\oplus \epsilon V\xrightarrow{\bar{s}} \oo_{T\times X} \xrightarrow{q_T} Q$ over $T\times X$.
Recall that $W$ is the geometric vector bundle associated to $\oo_\bp(1)\boxtimes V^{*[k]}$.
The composition
\begin{equation*}
  \gamma\colon V\hookrightarrow
  V\oplus \epsilon V
  \xrightarrow{\bar{s}}
  \oo_{T\times X} \xrightarrow{q_T} Q
\end{equation*}
factors through the $\oo_Z$ at the upper right corner of (\ref{def-f}),
 inducing the corresponding element in $\Hom(V,\oo_Z)$.
More concretely, the image of $a$ under (\ref{fibers-pot}) is
$$V\xrightarrow{a} \oo_X \twoheadrightarrow \oo_Z,$$
and the image of $\lambda$ is
$$V\xrightarrow{\alpha} I_Z \xrightarrow{\lambda} \oo_Z.$$
This completes the calculation of the map (\ref{fibers-pot}).

We now compute the kernel of (\ref{fibers-pot}). The morphism $\gamma$ is 0 if and only if $V\hookrightarrow V \oplus \epsilon V \xrightarrow{\bar{s}}\oo_{T\times X}$ factors through $E$ in (\ref{def-f}). Moreover, the composition $\epsilon V \xrightarrow{s} \epsilon \oo_X \to Q$ factors through $\oo_Z$ and is thus 0 since $s$ vanishes on $Z$, so in fact $\gamma$ is zero if and only if $\bar{s}$ factors through $E$. Since a map $\pi_2^* V \cong V \oplus \epsilon V \to E$ is a family of stable pairs over $T$, the kernel of (\ref{fibers-pot}) is isomorphic to the deformation space $\Hom(J^\bullet, I_Z)_0$.

Finally, we study the cokernel of (\ref{fibers-pot}).
The cokernel of $\Hom(V,\oo_X)/\mathbb{C} s \to \Hom(V,\oo_Z)$
is $\Ext^1(V,I_Z)$, following from the obvious exact sequence and the vanishing of the higher cohomology of $V^*$.
On the other hand, we have the following natural commutative diagram
\[\xymatrix{
    \Hom(I_Z,\oo_Z) \ar[d] \ar[r] & \Ext^1(I_Z,I_Z)_0\oplus H^1(\oo_X) \ar[d]\\
    \Hom(V,\oo_Z) \ar@{->>}[r] & \Ext^1(V,I_Z)
}\]
where the vertical arrows are both induced by $\alpha$.
The upper arrow identifies $\Hom(I_Z,\oo_Z)$ with $\Ext^1(I_Z,I_Z)_0$ according to Lemma~\ref{diff-tr}.
Because of the condition (\ref{tr-def-surj}), $H^1(\oo_X)$ is sent to zero by the right arrow.
The cokernel of (\ref{fibers-pot}) is then isomorphic to the cokernel of the right vertical arrow, which is $\Ext^1(J^\bullet, I_Z)_0$.

\bibliography{yinbangbib}{}
\bibliographystyle{alpha}

\end{document}